\documentclass[a4paper,11pt]{amsart}

\usepackage[english]{babel}

\usepackage[a4paper,left=30mm,right=30mm,top=35mm,bottom=35mm,marginpar=25mm]{geometry} 

\usepackage[T1]{fontenc}
\usepackage[utf8]{inputenc}
\usepackage{amsthm, amssymb, amsmath, dsfont, mathrsfs, bbold,amsfonts,}
\usepackage{mathtools}
\usepackage{graphicx}
\usepackage{stackrel}
\usepackage{enumerate}
\usepackage{xcolor}
\usepackage{enumitem}
\usepackage{esint}

\usepackage[colorlinks=true, citecolor=blue]{hyperref} 

\numberwithin{equation}{section}
\newtheorem{theorem}{Theorem}[section]
\newtheorem{lemma}[theorem]{Lemma}

\newtheorem{cor}[theorem]{Corollary}
\newtheorem{conj}[theorem]{Conjecture}

\theoremstyle{definition}
\newtheorem{defin}[theorem]{Definition}
\newtheorem{rem}[theorem]{Remark}

\newtheorem{notation}[theorem]{Notation}

\AtBeginDocument{
  \let\div\relax
  \DeclareMathOperator{\div}{div}
}

\DeclareMathOperator{\Capa}{cap}
\DeclareMathOperator{\dist}{dist}
\DeclareMathOperator{\diam}{diam}
\DeclareMathOperator{\graph}{graph}

\newcommand{\mres}{\mathbin{\vrule height 1.6ex depth 0pt width
0.13ex\vrule height 0.13ex depth 0pt width 1.3ex}}

\newcommand{\cp}[2]{\int_{#1}{\vert\nabla #2\vert^2}dx}

\title[The  sharp quantitative isocapacitary inequality]{The  sharp quantitative isocapacitary inequality}

\author[G. De Philippis]{Guido De Philippis}
\address{G.D.P.: SISSA, Via Bonomea 265, 34136 Trieste, Italy}
\email{guido.dephilippis@sissa.it}

\author[M. Marini]{Michele Marini}
\address{M.M: SISSA, Via Bonomea 265, 34136 Trieste, Italy}
\email{mmarini@sissa.it}

\author[E. Mukoseeva]{Ekaterina Mukoseeva}
\address{E.M: SISSA, Via Bonomea 265, 34136 Trieste, Italy}
\email{emukosee@sissa.it}

\subjclass[2010]{49R05 (47A75 49Q20)}
\keywords{Isocapacitary inequality, stability estimates, Fraenkel asymmetry}

\begin{document}

\begin{abstract}
We prove a sharp quantitative form of the classical isocapacitary inequality. Namely,  we show that the difference between the capacity of a set and that of a ball with the same volume bounds the square of the Fraenkel asymmetry of the set. This provides a positive answer to a conjecture of Hall, Hayman, and Weitsman (J. d' Analyse Math. '91).
\end{abstract}

\maketitle

\section{Introduction}
	
\subsection{Background}
Let $\Omega\subset\mathds{R}^N$, $N\ge 3$ be an open set. We define the \emph{ absolute capacity} of $\Omega$ as \begin{equation}\label{capdef}
\Capa(\Omega)=\inf_{u\in C_c^\infty(\mathds{R}^N)}\left\{\int_{\mathds{R}^N}{\vert\nabla u\vert^2}dx: u\ge1\text{ on }\Omega\right\}.
\end{equation}
Moreover, for $\Omega\subset B_R$  (\(B_R\) the ball of radius \(R\) centered at the origin) we denote by  $\Capa_R(\Omega)$ the \emph{relative capacity} of $\Omega$ with respect to $B_R$ defined as
\begin{equation}\label{relcapdef}
\Capa_R(\Omega)=\inf_{u\in C_c^\infty(B_R)}\left\{\int_{B_R}{\vert\nabla u\vert^2}dx: u\ge 1\text{ on }\Omega\right\}.
\end{equation}	
It  is easy to see that  for problem \eqref{capdef} (resp. \eqref{relcapdef}) there exists  a unique  function\footnote{Here and in the sequel, \(D^{1,2}(\mathds{R}^N)\) denotes  the closure of \(C^\infty_c(\mathds{R}^N)\) with respect to the homogeneous  Sobolev norm: 
\[
 \|u\|_{ \dot{W}^{1,2}}:=\| \nabla u\|_{L^2},
\]
 see~\cite[Section 4.7]{EvansGariepy15} and~\cite[Chapter 8]{LiebLoss97}}  
 \(u\in D^{1,2}(\mathds{R}^N)\)  (resp. \(u_R\in W^{1,2}_{0}(B_R)\)) called  {\it capacitary potential} of $\Omega$ such that
\[
\int_{\mathds{R}^N} |\nabla u|^2=\Capa(\Omega) \qquad \Bigl(\text{resp.} \int_{B_R} |\nabla u_R|^2=\Capa_R(\Omega) \Bigr).
\]

%
%
Moreover, they satisfy the  Euler-Lagrange equations:			
\begin{equation*}
\begin{cases}
\Delta u=0\text{ in }\overline{\Omega}^c \\
u=1\text{ on }\partial \Omega \\
u(x)\rightarrow 0\text{ as }x\rightarrow 0
\end{cases}
\qquad
\begin{cases}
\Delta u_R=0\text{ in }B_R\setminus \overline{\Omega} \\
u_R=1\text{ on }\partial \Omega \\
u_R=0\text{ on }\partial B_R.
\end{cases}
\end{equation*}

The well-known {\it isocapacitary inequality} (resp. {\it relative isocapacitary inequality}) asserts that, among all sets with given volume, balls (resp. ball centered at the origin) have the smallest possible capacity, namely
\begin{equation}\label{iso1}
\Capa(\Omega)-\Capa(B_r)\ge 0  \qquad  \text{(resp. }\mathrm{cap_R}(\Omega)-\mathrm{cap_R}(B_r)\ge 0).
\end{equation}
Here $r$ is such that $|B_r|=|\Omega|$, where $|\cdot|$ denotes the Lebesgue measure.\\

The proof is an easy combination of  {\it Schwarz symmetrization} with  {\it P\'olya-Szeg\"o principle}. Indeed, let $\Omega$ be an open set and let $u$ be its capacitary potential. Schwarz symmetrization provides us with a radially symmetric function $u^*$ such that, for every $t\in\mathds R$,
\begin{equation}\label{equim}
\left\vert\{x: u(x)>t\}\right\vert=\left\vert\{x: u^*(x)>t\}\right\vert.
\end{equation}
We use $u^*$ as a test function for the set $\{x\,:\, u^*(x)=1\}=B_r$  and we note that  \eqref{equim} yields that $|B_r|=|\Omega|$). Hence  
\[
\Capa(B_r)\leq\int_{\mathds{R}^N}{\vert\nabla u^*\vert^2}dx\le\int_{\mathds{R}^N}{\vert\nabla u\vert^2}dx
=\Capa(\Omega) \qquad |\Omega|=|B_r|,
\]
where the second inequality follows by  P\'olya-Szeg\"o principle. The very same argument  applies to  the relative isocapacitary inequality.

Inequalities \eqref{iso1} are rigid, namely, equality is attained only when $\Omega$ coincides with a ball, up to a set of zero capacity. For the case of the relative isocapacitary inequality $\Omega$ must  instead coincide with a centered ball, since this latter notion of capacity is not invariant under translations.

\medskip

It is natural to wonder whether these inequalities are also stable, that is $\Omega\to B_r$, whenever $\Capa(\Omega)\to\Capa(B_r)$. In particular, one aims to a (possibly sharp) quantitative enhancement of inequalities \eqref{iso1} by replacing their right-hand side with some function of the distance of  \(\Omega\) from the set of balls.

As we shall explain in the following sections, the answer is positive, and a good choice of distance is the so-called {\it Fraenkel asymmetry}.

	\begin{defin}
		Let $\Omega$ be an open set. The Fraenkel asymmetry of \(\Omega\),  $\mathcal{A}(\Omega)$, is defined as:
		\begin{equation*}
			\mathcal{A}(\Omega)=\inf\left\{\frac{\vert\Omega\Delta B\vert}{\vert B\vert}\,:\, B \text{ is a ball with the same volume as }\Omega\right\}.
		\end{equation*}
	\end{defin}	
To the best of our knowledge, the first results in this direction  appeared in \cite{HHW} where they considered the  case of simply connected planar sets 
\footnote{Note that for \(N=2\) the infimum  \eqref{capdef} is \(0\) and one has to use the notion of logarithmic capacity.} 
and of convex sets in general dimension.  In the same paper the authors conjecture the validity of the following  inequality:
	\begin{conj} [\cite{HHW}]\label{conj:HHWN}
		Let \(N\ge 3\).  There exists a constant $c=c(N)$ such that for any open set $\Omega$ the following inequality holds:
		\begin{equation*}
			\frac{\Capa(\Omega)-\Capa(B_r)}{r^{N-2}}\geq c\mathcal{A}(\Omega)^2.
		\end{equation*}
	\end{conj}
	
Note that by testing the inequality on ellipsoids with eccentricity \(\varepsilon\) one easily sees that the exponent \(2\) can not be replaced by any smaller number. 

A positive answer to the above conjecture in dimension \(2\) has been given by Hansen and Nadirashvili  in \cite{HN}. For general dimension,  the best known result is due to Fusco, Maggi, and Pratelli in \cite{nonsharp} where they prove the following:

 \begin{theorem} [\cite{nonsharp}]\label{t:fmp}
		There exists a constant $c=c(N)$ such that, for any open set $\Omega$
		\begin{equation*}
			\frac{\Capa(\Omega)-\Capa(B_r)}{r^{N-2}}\geq c\mathcal{A}(\Omega)^4.
		\end{equation*}
	\end{theorem}
	
In this paper we  provide a positive answer to Conjecture \ref{conj:HHWN} in every dimension and to its version for the relative capacity.	

	\subsection{Main result}
	
	The following is the main result of the paper, note that by the scaling \(\Capa(\lambda \Omega)=\lambda^{N-2}\Capa(\Omega)\), we can also get the analogous result for  $\Omega$ with  arbitrary volume.

	\begin{theorem} \label{mainthmfrnkl}
		Let $\Omega$ be an open set such that $\vert\Omega\vert=\vert B_1\vert$.
		Then
		\begin{enumerate}[label=\textup{(\Alph*)}]
			\item if $\Omega$ is contained in $B_R$, there exists a constant $c_1=c_1(N,R)$ such that the following inequality holds:
				\begin{equation*}
					\Capa_R(\Omega)-\Capa_R(B_1)\geq c_1(N,R)|\Omega\Delta B_1|^2.
				\end{equation*}
			\item there exists a constant $c_2=c_2(N)$ such that the following inequality holds:
				\begin{equation*}
					\Capa(\Omega)-\Capa(B_1)\geq c_2(N)\mathcal{A}(\Omega)^2.
				\end{equation*}
				
		\end{enumerate}
	\end{theorem}

		Note that in the above theorem,  in the case of the absolute capacity one  bound  the distance of \(\Omega\) from the set of balls, while in the case of the relative capacity one bounds the distance from the ball \emph{centered at the origin} but the constant is \(R\) dependent. Indeed in the former case all balls have the same capacity (due to the translation invariance of the problem) and thus in order to obtain a quantitative improvement, one has to measure the distance from the set of \emph{all minimizers}. On the contrary, for the relative capacity,   the ball centered at the origin  is the only minimizer.  Since 
		\[
		\lim_{R\to +\infty}\Capa_R(\Omega)=\Capa(\Omega),
		\]
		it is clear that the constant in (B) above needs to depend on  \(R\). This can also be inferred by the study of the linearized problem, see Section  \ref{sec:lin} below. We also remark that,  as it will be clear from the proof, in the case of the relative capacity one can replace  \(|\Omega\Delta B_1|\) with the bigger quantity \(\alpha_R(\Omega)\) defined in Section \ref{sec:lin} below.


	\subsection{Strategy of the proof and  structure  of the paper}
	Since the isocapacitary inequality is a consequence of the isoperimetric inequality, a reasonable strategy to obtain a quantitative improvement would be to rely on a quantitative isoperimetric inequality. This was indeed the strategy used in \cite{nonsharp} where they rely on the quantitative isoperimetric inequality established  in~\cite{FuscoMaggiPratelli08}. However, although the inequality proved in ~\cite{FuscoMaggiPratelli08} is sharp, in order to combine it with the Schwarz symmetrization procedure,  it seems unavoidable to lose some exponent and to obtain a result in line with the one in \cite{nonsharp}.
		
	Here we instead rely on the techniques developed by the first author with Brasco and Velichkov in \cite{fkstab} to obtain a quantitative form of the  Faber-Krahn inequality (see also~\cite{BrascoDe-Philippis17} and references therein for a survey on these type of results). The proof is  based on  the Selection Principle,  introduced
	by Cicalese and Leonardi in \cite{CL} to give  a new proof of the sharp quantitative isoperimetric inequality,  combined with the regularity estimates for free boundary problems obtained by Alt and Caffarelli in \cite{AC}. As in \cite{fkstab}, one of the key technical tools is to replace the Fraenkel asymmetry (which roughly resembles a \(L^1\) type norm) with a smoother (and stronger)  version inspired by the distance among sets first used by Almgren Taylor and Wang in~\cite{AlmgrenTaylorWang93} which resembles an \(L^2\) type norm, see Section \ref{sec:lin} for the exact definition.

We conclude this  introduction by giving an account of the main steps of the proof and of the structure of the paper:

The main  step consists in proving  Theorem \ref{mainthmfrnkl} for a priori  bounded sets  in the regime of small asymmetry. Arguing  by contradiction  one obtains a sequence of sets contradicting the stability inequality with any given constant \(c>0\).   In Sections \ref{penpb} and \ref{s:ex}  we use this sequence to construct an improved contradicting  sequence  which solves a variational problem.
	
	In Section \ref{reg}, we exploit the regularity theory of \cite{AC} to show that this new sequence consists of  smooth {\it nearly spherical sets}, for which the desired estimate is proved in Section \ref{fuglede}, via a Fuglede type computation~\cite{Fuglede89}. 		
	In Section \ref{redtobdd}, we show how one can reduce to a priori bounded domains for the case of the absolute capacity. Eventually, in Section \ref{s:proof} we combine all the steps to prove Theorem \ref{mainthmfrnkl}.

 \subsection*{Acknowledgements}
The work of the authors  is supported by the INDAM-grant ``Geometric Variational Problems".

	\section{Fuglede's computation}	\label{fuglede}
	
As explained in the introduction it is convenient to introduce a smoothed version of the Fraenkel asymmetry. Roughly speaking, while \(\mathcal A(\Omega)\) represents an \(L^1\) norm, \(\alpha(\Omega)\) represents an \(L^2\) norm, see \ref{asymnrlysphrsets} in Lemma \ref{propasym} below and the discussion in  \cite[Introduction]{fkstab}.

	\begin{defin}
		Let $\Omega$ be an open set in $\mathds{R}^N$.
		Then we define the asymmetry $\alpha$ in the following way:
		\begin{enumerate}[label=(\Alph*)] 
			
			\item	
				\begin{equation*}
					\alpha_R(\Omega)=\int_{\Omega\Delta B_1}\big\vert 1-\vert x\vert\big\vert dx;
				\end{equation*}
			\item 
				\begin{equation*}
					\alpha(\Omega)=\int_{\Omega\Delta B_1(x_\Omega)}\big\vert 1-\vert x-x_\Omega\vert\big\vert dx.
				\end{equation*}
		\end{enumerate}
		Here $x_\Omega$ denotes the barycenter of $\Omega$, namely $x_\Omega=\fint_{\Omega}{x}dx$.
	\end{defin}	
	
	Since most of the argument will  be similar for the relative and for the absolute capacity, let us also introduce the following notational convention:
	
	\begin{notation}\label{notaz} Whenever possible, we will write $\alpha_*$,$\Capa_*$, etc. instead of $\alpha$/$\alpha_R$,
	$\Capa$/$\Capa_R$ or other notions that will come along. The convention is that $*$ denotes the same thing ($R$
	or the absence of it) throughout the equation or the computation where it appears. 
	\end{notation}	
	
	The next Lemma collects the main properties of \(\alpha\), the proof is identical to the one of  \cite[Lemma 4.2]{fkstab} and it is left to the reader.
	\begin{lemma} \label{propasym}
	Let \(\Omega \subset \mathds R^n\), then 
		\begin{enumerate}[label=\textup{(\roman*)}]
			\item \label{compasym} There exists a constant $c=c(N)$ such that
				\begin{enumerate}[label=\textup{(\Alph*)}] 
					\item
						\begin{equation*}
							\alpha_R(\Omega)\geq c\vert\Omega\Delta B_1\vert^2
						\end{equation*}
						for any open set $\Omega\subset B_R$;
					\item
						\begin{equation*}
							\alpha(\Omega)\geq c\vert\Omega\Delta B_1(x_\Omega)\vert^2
						\end{equation*}
						for any open set $\Omega$.
				\end{enumerate}		
			\item \label{asymlip} There exists a constant $C=C(R)$ such that	
				\begin{equation*}
					\vert\alpha_*(\Omega_1)-\alpha_*(\Omega_2)\vert\leq C\vert\Omega_1\Delta \Omega_2\vert
				\end{equation*}
				for any $\Omega_1,\Omega_2\subset B_R$. In particular, if $1_{\Omega_k}\rightarrow 1_\Omega$ in $L^1(B_R)$
				then  $\alpha_*(\Omega_k)\rightarrow\alpha_*(\Omega)$.
			\item \label{asymnrlysphrsets}
			There exist constants $C=C(N)$, $\delta=\delta(N)$ such that
			for every nearly spherical set (see Definition \ref{d:ns} below) $\Omega$ with $\Vert\phi\Vert_{\infty}\leq\delta$ 
			(and $x_\Omega=0$ in the case of $\alpha$)
			\begin{equation*}
				\alpha_*(\Omega)\leq C\Vert\phi\Vert_{L^2(\partial B_1)}^2.
			\end{equation*}	
		\end{enumerate}
	\end{lemma}

		We now prove the validity of the  quantitative isocapacitary inequality for sets close to the unit ball.
		More precisely, we are going to prove Theorem \ref{mainthmfrnkl} for nearly spherical sets which are defined below. The proof is based on second variation argument as in \cite{Fuglede89}.		

		\begin{defin}\label{d:ns}
		  An open bounded set $\Omega\subset\mathds{R}^N$ is called nearly spherical
			of class $C^{2,\gamma}$ parametrized by $\varphi$, if there exists
			$\varphi\in C^{2,\gamma}$ with $\Vert\varphi\Vert_{L^\infty}<\frac{1}{2}$ such that
			\begin{equation*}
				\partial\Omega=\{(1+\varphi(x))x: x\in\partial B_1\}.
			\end{equation*}	
		\end{defin}

Let us also introduce the following definition:
		\begin{defin}
			Given a function $\varphi:\partial B_1\rightarrow\mathds{R}$ we define
			\begin{enumerate}[label=(\Alph*)] 
			\item
			$H_R(\varphi)\in W^{1,2}_0(B_R)$ as the solution to
			\begin{equation*}
				\begin{cases}
					\Delta H_R(\varphi)=0\text{ in }B_R\backslash B_1 \\
					H_R(\varphi)=\varphi\text{ on }\partial B_1 \\
					H_R(\varphi)=0\text{ on }\partial B_R
				\end{cases}
			\end{equation*}	
			\item
			$H(\varphi)\in D^{1,2}(\mathds{R}^N)$ as the solution to
			\begin{equation*}
				\begin{cases}
					\Delta H(\varphi)=0\text{ in }B_1^c \\
					H(\varphi)=\varphi\text{ on }\partial B_1 \\
					H(\varphi)(x)\rightarrow 0\text{ as }x\rightarrow\infty
				\end{cases}
			\end{equation*}	
			\end{enumerate}
		\end{defin}

		\subsection{Second variation}
		
		We now  compute the second order expansion of the  capacity of a nearly spherical set. Note that the remainder term is multiplied by a higher order norm. This is precisely the reason why  we will need to use the Selection Principle in the proof of Theorem \ref{mainthmfrnkl}.
		
		\begin{lemma} \label{taylor}
			Given $\gamma\in(0,1]$, there exists $\delta=\delta(N,\gamma)>0$
			and a modulus of continuity $\omega$ such that for every nearly spherical set $\Omega$
			parametrized by $\varphi$ with $\Vert\varphi\Vert_{C^{2,\gamma}(\partial B_1)}<\delta$
			and $\vert\Omega\vert=\vert B_1\vert$, we have
			\begin{equation*}
				\Capa_*(\Omega)\geq \Capa_*(B_1)+\frac{1}{2}\partial^2 \Capa_*(B_1)[\varphi,\varphi]
				-\omega(\Vert\varphi\Vert_{C^{2,\gamma}})\Vert\varphi\Vert^2_{H^\frac{1}{2}(\partial B_1)},
			\end{equation*}
			where
			\begin{enumerate}[label=(\Alph*)] 
			\item
			\begin{equation*}
				\partial^2 \Capa_R(B_1)[\varphi,\varphi]:=2\frac{(N-2)^2}{1-R^{-(N-2)}}\left(\int_{B_R\backslash B_1}{\vert\nabla H_R(\varphi)\vert^2}dx
				-(N-1)\int_{\partial B_1}{\varphi^2}d\mathcal{H}^{N-1}\right);
			\end{equation*}
			\item
			\begin{equation*}
				\partial^2 \Capa(B_1)[\varphi,\varphi]:=2(N-2)^2\left(\int_{B_1^c}{\vert\nabla H(\varphi)\vert^2}dx
				-(N-1)\int_{\partial B_1}{\varphi^2}d\mathcal{H}^{N-1}\right).
			\end{equation*}
			\end{enumerate}
		\end{lemma}		
		
		To prove it, let us  first introduce a technical lemma.
		
		\begin{lemma}\label{lm:vectorfield}
			Given $\gamma\in(0,1]$ there exists $\delta=\delta(N,\gamma)>0$
			and a modulus of continuity $\omega$ such that for every nearly spherical set $\Omega$
			parametrized by $\varphi$ with $\Vert\varphi\Vert_{C^{2,\gamma}(\partial B_1)}<\delta$
			and $\vert\Omega\vert=\vert B_1\vert$, we can find an autonomous vector field $X_\varphi$
			for which the following holds true:
			\begin{enumerate}[label=\textup{(\roman*)}]
				\item $\div{X_\varphi}=0$ in a $\delta$-neighborhood of $\partial B_1$;
				\item $X_\varphi=0$ outside  a  $2\delta$-neighborhood of $\partial B_1$;
				\item if $\Phi_t:=\Phi(t,x)$ is the flow of $X_\varphi$, i.e.
					\begin{equation*}
					  \partial_t\Phi_t=X_\varphi(\Phi_t),\qquad \Phi_0(x)=x,
					\end{equation*}
					then $\Phi_1(\partial B_1)=\partial\Omega$ and $\vert\Phi_t(B_1)\vert=\vert B_1\vert$ for all $t\in[0,1]$;
				\item 
					\begin{itemize}
						\item $\Vert\Phi_t-Id\Vert_{C^{2,\gamma}}\leq\omega(\Vert\varphi\Vert_{C^{2,\gamma}(\partial B_1)})$ for every $t\in[0,1]$,
						\medskip
						
						\item $\Vert\varphi-(X_\varphi\cdot\nu_{B_1})\Vert_{H^{\frac{1}{2}}(\partial B_1)}\leq\omega(\Vert\varphi\Vert_{L^\infty(\partial B_1)})\Vert\varphi\Vert_{H^{\frac{1}{2}}(\partial B_1)}$,
						
						\medskip
						
						\item $(X\cdot x)\circ\Phi_t-X\cdot\nu_{B_1}=(X\cdot\nu_{B_1})\psi_t$, $x\in\partial B_1$, 
						
						where $\Vert\psi_t\Vert_{C^{2,\gamma}(\partial B_1)}\leq\omega(\Vert\varphi\Vert_{C^{2,\gamma}(\partial B_1)})$.
					\end{itemize}
			\end{enumerate}
			
		\end{lemma}
		\begin{proof}
		  Take the same vector field as in Appendix of \cite{fkstab} and multiply it by a cut-off function.
		\end{proof}

\begin{proof}[Proof of Lemma \ref{taylor}]		
		
		Now set $\Omega_t:=\Phi_t(B_1)$ and let $u_t$ be the capacitary potential of $\Omega_t$. We define 
		\begin{equation*}
			c_*(t):=\Capa_*(\Omega_t)=
			\begin{cases}
		  		\int_{B_R\backslash\Omega_t}{\vert\nabla u_t\vert^2 dx}\text{ in the case of relative capacity};\\
		  		\int_{\Omega_t^c}{\vert\nabla u_t\vert^2 dx}\text{ in the case of full capacity}.
		  	\end{cases}
		\end{equation*}
		
		It is easy to see that $t\mapsto u_t$ is differentiable, see~\cite{Dambrine02},  and that its derivative $\dot{u_t}$ satisfies
		\begin{enumerate}[label=(\Alph*)]
		\item
		\begin{equation*}
			\begin{cases}
				\Delta \dot{u_t}=0\text{ in }B_R\backslash \Omega_t \\
				\dot{u_t}=-\nabla u_t \cdot X_\varphi \text{ on }\partial \Omega_t\\ 
				\dot{u_t}=0\text{ on }\partial B_R 
			\end{cases}
		\end{equation*}	
		\item
		\begin{equation*}
			\begin{cases}
				\Delta \dot{u_t}=0\text{ in }\Omega_t^c \\
				\dot{u_t}=-\nabla u_t \cdot X_\varphi \text{ on }\partial \Omega_t \\
				\dot{u_t}(x)\rightarrow 0\text{ as }x\rightarrow 0
			\end{cases}
		\end{equation*}	
		\end{enumerate}			
		Using Hadamard formula, we compute:
		\begin{equation*}
		  \frac{1}{2}c'_R(t)=\int_{B_R\backslash\Omega_t}{\nabla u_t\cdot\nabla\dot{u_t}}dx
			+\frac{1}{2}\int_{\partial\Omega_t}{\vert\nabla u_t\vert^2X_\varphi\cdot\nu_{\Omega_t}}d\mathcal{H}^{N-1},
		\end{equation*}	
		where $\nu_{\Omega_t}$ is the \emph{inward} normal to $\partial\Omega_t$.
		Now we recall that $u_t$ is harmonic in $B_R\backslash\Omega_t$ and we use the boundary conditions for $\dot{u_t}$ to get	
		\begin{equation*}
			\begin{aligned}
				\frac{1}{2}c'_R(t)&=\int_{B_R\backslash\Omega_t}{\div{(\dot{u_t}\nabla u_t)}}dx
				+\frac{1}{2}\int_{\partial\Omega_t}{\vert\nabla u_t\vert^2X_\varphi\cdot\nu_{\Omega_t}}d\mathcal{H}^{N-1}\\
				&=\int_{\partial\Omega_t}{\dot{u_t}\nabla u_t\cdot\nu_{\Omega_t}}d\mathcal{H}^{N-1}
				+\frac{1}{2}\int_{\partial\Omega_t}{\vert\nabla u_t\vert^2X_\varphi\cdot\nu_{\Omega_t}}d\mathcal{H}^{N-1}\\
				&=\int_{\partial\Omega_t}{(-\nabla u_t\cdot X_\varphi)\nabla u_t\cdot\nu_{\Omega_t}}d\mathcal{H}^{N-1}
				+\frac{1}{2}\int_{\partial\Omega_t}{\vert\nabla u_t\vert^2X_\varphi\cdot\nu_{\Omega_t}}d\mathcal{H}^{N-1}.\\
			\end{aligned}	
		\end{equation*}	
		We know that $u_t$ is identically $1$ on $\partial\Omega_t$ and smaller than \(1\) outside, hence (recall that \(\nu_{\partial \Omega_t}\) denotes the ineer normal)
		\begin{equation} \label{grad_u_on_bndry}
			\nabla u_t=\vert\nabla u_t\vert\nu_{\partial\Omega_t} \text{ on }\partial\Omega_t.
		\end{equation}
		Therefore, 
		\begin{equation*}
			\begin{aligned}
				\frac{1}{2}c'_R(t)&=\int_{\partial\Omega_t}{-\vert\nabla u_t\vert^2 X_\varphi\cdot\nu_{\Omega_t}}d\mathcal{H}^{N-1}
				+\frac{1}{2}\int_{\partial\Omega_t}{\vert\nabla u_t\vert^2X_\varphi\cdot\nu_{\Omega_t}}d\mathcal{H}^{N-1}\\
				&=-\frac{1}{2}\int_{\partial\Omega_t}{\vert\nabla u_t\vert^2X_\varphi\cdot\nu_{\Omega_t}}d\mathcal{H}^{N-1}
				=-\frac{1}{2}\int_{B_R\backslash\Omega_t}{\div{(\vert\nabla u_t\vert^2X_\varphi)}}dx.\\
			\end{aligned}	
		\end{equation*}	
		We proceed now with the second derivative, using again Hadamard's formula and recalling that $X$ is autonomous
		and divergence-free in a neighborhood of $\partial B_1$ (hence, on $\partial\Omega_t$).
		\begin{equation*}
			\begin{aligned}
				\frac{1}{2}c''_R(t)&=-\frac{1}{2}\int_{B_R\backslash\Omega_t}\div\Bigl(\frac{\partial}{\partial t}\vert\nabla u_t\vert^2X_\varphi\Bigr)dx
				-\frac{1}{2}\int_{\partial\Omega_t}{\div{(\vert\nabla u_t\vert^2X_\varphi)}(X_\varphi\cdot\nu_{\Omega_t})}d\mathcal{H}^{N-1}\\
				&=-\int_{\partial\Omega_t}{(\nabla u_t\cdot\nabla\dot{u_t})X_\varphi\cdot\nu_{\Omega_t}}d\mathcal{H}^{N-1}
				-\frac{1}{2}\int_{\partial\Omega_t}{(\nabla\vert\nabla u_t\vert^2\cdot X_\varphi)}(X_\varphi\cdot\nu_{\Omega_t})d\mathcal{H}^{N-1}\\
				&=\int_{\partial\Omega_t}{\dot{u_t}\nabla\dot{u_t}\cdot\nu_{\Omega_t}}d\mathcal{H}^{N-1}
				-\int_{\partial\Omega_t}{(X_\varphi\cdot\nu_{\Omega_t})(\nabla^2u_t[\nabla u_t]\cdot X_\varphi)}d\mathcal{H}^{N-1}\\
				&=\int_{B_R\backslash\Omega_t}{\vert\nabla\dot{u_t}}\vert^2dx
				-\int_{\partial\Omega_t}{(X_\varphi\cdot\nu_{\Omega_t})(\nabla^2u_t[\nabla u_t]\cdot X_\varphi)}d\mathcal{H}^{N-1}\\
			\end{aligned}	
		\end{equation*}	
		Note that in the second to last equality we have used \eqref{grad_u_on_bndry} and the boundary condition for \(\dot{u_t}\). Now  since \(u_t\) is constant on \(\partial \Omega_t\),
		we get
		\begin{equation*}
			0=\Delta u_t=\vert\nabla u_t\vert\mathscr{H}_{\partial\Omega_t}+\nabla^2[\nu_{\Omega_t}]\cdot\nu_{\Omega_t} \text{ on }\partial\Omega_t,
		\end{equation*}
		where $\mathscr{H}_{\partial\Omega_t}$ is the mean curvature of $\partial\Omega_t$ with respect to the inward normal to $\partial\Omega_t$.
		Taking this into account and denoting $X^\tau=X_\varphi-(X_\varphi\cdot\nu_{\Omega_t})\nu_{\Omega_t}$ on $\partial\Omega_t$, we get
		\begin{equation*}
			\begin{aligned}
				\frac{1}{2}c''_R(t)&=\int_{B_R\backslash\Omega_t}{\vert\nabla\dot{u_t}}\vert^2dx
				\\
				&\qquad-\int_{\partial\Omega_t}{(X_\varphi\cdot\nu_{\Omega_t})(\nabla^2u_t[\vert\nabla u_t\vert\nu_{\Omega_t}]\cdot ((X_\varphi\cdot\nu_{\Omega_t})\nu+X^\tau))}d\mathcal{H}^{N-1}\\
				&=\int_{B_R\backslash\Omega_t}{\vert\nabla\dot{u_t}}\vert^2dx
				+\int_{\partial\Omega_t}{(X_\varphi\cdot\nu_{\Omega_t})^2\vert\nabla u_t\vert^2\mathscr{H}_{\partial\Omega_t}}d\mathcal{H}^{N-1}
				\\		
 &\qquad-\int_{\partial\Omega_t}{(X_\varphi\cdot\nu_{\Omega_t})(\nabla^2u_t[\nabla u_t]\cdot X^\tau)}d\mathcal{H}^{N-1}.\\
			\end{aligned}	
		\end{equation*}	
		Now we wish to calculate $c''_R(0)$. We use that
		\begin{itemize}
			\item $\mathscr{H}_{\partial B_1}=-(N-1)$;
			\item $X^\tau=0$ on $\partial B_1$;
			\item $u_0=u_{B_1}=\frac{\vert x\vert^{-(N-2)}-R^{-(N-2)}}{1-R^{-(N-2)}}$ in $B_R\backslash B_1$;
			\item $\dot{u_0}=H_R(-X_\varphi\cdot\nabla u_0)$.		
		\end{itemize}

		\begin{equation*}
			\begin{aligned}
				\frac{1}{2}c''_R(0)&=\int_{B_R\backslash B_1}{\vert\nabla H_R(-X_\varphi\cdot\nabla u_0)\vert^2}dx
				-(N-1)\int_{\partial B_1}{(X_\varphi\cdot\nu_{B_1})^2\vert\nabla u_0\vert^2}d\mathcal{H}^{N-1}\\
				&=\frac{(N-2)^2}{1-R^{-(N-2)}}\left(\int_{B_R\backslash B_1}{\vert\nabla H_R(X_\varphi\cdot\nu_{B_1})\vert^2}dx
				-(N-1)\int_{\partial B_1}{(X_\varphi\cdot\nu_{B_1})^2}d\mathcal{H}^{N-1}\right)\\
				\end{aligned}	
		\end{equation*}	

		As for the case of full capacity, the same computations apply with minor changes, obtaining 
		\begin{equation*}
			\frac{1}{2}c''(0)=(N-2)^2\left(\int_{B_1^c}{\vert\nabla H_R(X_\varphi\cdot\nu_{B_1})\vert^2}dx
				-(N-1)\int_{\partial B_1}{(X_\varphi\cdot\nu_{B_1})^2}d\mathcal{H}^{N-1}\right),
		\end{equation*}	
which formally corresponds to sending \(R\to \infty\) in the formula for \(c''_R\). Since balls minimize the capacity we also have that \(c_*'(0)=0\). Writing 
\[
\Capa_*(\Omega)=c_*(1)=c_*(0)+\frac{1}{2}c_*''(0)+\int_0^1(1-t)(c_*''(t)-c_*''(0))dt\,,
\]
one can now exploit Lemma \ref{lm:vectorfield} and perform the very same computations as in \cite[Lemma A.2]{fkstab} to conclude.
\end{proof}
		
\subsection{Inequality for nearly spherical sets}\label{sec:lin}
We now establish a quantitative inequality for nearly spherical sets in the spirit of those established by Fuglede in~\cite{Fuglede89}, compare with    \cite[Section 3]{fkstab}.

\begin{theorem} \label{mainthmnrlysphr}
	There exists $\delta=\delta(N), c=c(N,R)$ ($c=c(N)$ for the capacity in $\mathds{R}^N$) such that if $\Omega$ is a nearly spherical set of class $C^{2,\gamma}$
	parametrized by $\varphi$ with $\Vert\varphi\Vert_{C^{2,\gamma}}\leq\delta,\vert\Omega\vert=\vert B_1\vert$
	(and $x_\Omega=0$ for the case of the capacity in $\mathds{R}^N$), then
	\begin{equation*}
		\Capa_*(\Omega)-\Capa_*(B_1)\geq c\Vert\varphi\Vert^2_{H^\frac{1}{2}(\partial B_1)},
	\end{equation*}
	where 
	\[
	\Vert\varphi\Vert^2_{H^\frac{1}{2}(\partial B_1)}:=\int_{\partial B_1}\varphi^2 d\mathcal{H}^{N-1}+\int_{B^c_1}\vert\nabla H_*(\varphi)\vert^2 dx,
	\]
	where the second integral is intended on \(B_R\setminus B_1\) if \(*=R\).
\end{theorem}
\begin{rem}
	Note that by Lemma \ref{propasym} \ref{compasym},\ref{asymnrlysphrsets} this theorem gives us Theorem \ref{mainthmfrnkl} for nearly spherical sets.
\end{rem}
\begin{proof}
	 We essentially repeat the proof of the Theorem 3.3 in \cite{fkstab}. First, we show that $\int_{\partial B_1}{\varphi}$ is small. Indeed, we know that
	\begin{equation*}
	\begin{aligned}
		\vert B_1\vert=\vert\Omega\vert&=\int_{\partial B_1}{\frac{\left(1+\varphi(x)\right)^N}{N}}d\mathcal{H}^{N-1}\\
		&=\vert B_1\vert+\int_{\partial B_1}{\varphi(x)}d\mathcal{H}^{N-1}+\int_{\partial B_1}{\sum_{i=2}^N{N\choose i}\frac{\varphi(x)^i}{N}}d\mathcal{H}^{N-1}.
	\end{aligned}	
	\end{equation*}
	Hence,
	\begin{equation*}
	\begin{split}
		\left\vert\int_{\partial B_1}{\varphi(x)}d\mathcal{H}^{N-1}\right\vert&=\left\vert\int_{\partial B_1}{\sum_{i=2}^N{N\choose i}\frac{\varphi(x)^i}{N}}d\mathcal{H}^{N-1}\right\vert
		\\
		&\leq C(N)\int_{\partial B_1}{\varphi(x)^2}d\mathcal{H}^{N-1}\leq C(N)\delta\|\varphi\|_{L^2}.
		\end{split}
	\end{equation*}

	Moreover, for the case of the absolute capacity, also $\int_{\partial B_1}{x_i\varphi}$ is small.
	Indeed, using that the barycenter of $\Omega$ is at the origin, we get 
	\begin{equation*}
		\left\vert\int_{\partial B_1}{x_i\varphi(x)}d\mathcal{H}^{N-1}\right\vert\leq\int_{\partial B_1}{\sum_{i=2}^N{N\choose i}\left\vert\frac{\varphi(x)^i}{N+1}\right\vert} d\mathcal{H}^{N-1}\leq C(N)\delta\|\varphi\|_{L^2}.
	\end{equation*}

Let us define
	\begin{enumerate}[label=(\Alph*)]
	\item
	$$\mathcal{M}^R_\delta:=\{\xi\in H^\frac{1}{2}(\partial B_1):
	\left\vert\int_{\partial B_1}\xi d\mathcal{H}^{N-1}\right\vert\leq\delta\|\xi\|_{H^{1/2}}\};$$ 
	\item
	$$\mathcal{M}_\delta:=\Bigl\{\xi\in H^\frac{1}{2}(\partial B_1):
	\left\vert\int_{\partial B_1}\xi d\mathcal{H}^{N-1}\right\vert+\left\vert\int_{\partial B_1}x\xi d\mathcal{H}^{N-1}\right\vert\leq\delta\|\xi\|_{H^{1/2}}\Bigr\},$$ 
	\end{enumerate}	 
	 and note that, since \(\|\xi\|_{L^2}\le \|\xi\|_{H^{1/2}}\), we have just proved that $\varphi$ belongs to $\mathcal{M}^*_{C\delta}$.	
	
	By Lemma \ref{taylor}, for $\delta$ small enough we have
	\begin{equation}
		\Capa_*(\Omega)-\Capa_*(B_1)\geq\frac{1}{2}\partial^2 \Capa_*(B_1)[\varphi,\varphi]
		-\omega(\Vert\varphi\Vert_{C^{2,\gamma}})\Vert\varphi\Vert^2_{H^\frac{1}{2}(\partial B_1)}.
	\end{equation}
So, it is enough to check that 
	\begin{equation*}
		\partial^2 \Capa_*(B_1)[\xi,\xi]\geq c\Vert\xi\Vert^2_{H^\frac{1}{2}(\partial B_1)}, \text{ for every $\xi\in\mathcal{M}^*_{\delta}$}
	\end{equation*}
	for small $\delta$.
	
	\emph{Step 1: linearized problem.} First, we show that
	\begin{equation*}
		\partial^2 \Capa_*(B_1)[\xi,\xi]\geq c\Vert\xi\Vert^2_{H^\frac{1}{2}(\partial B_1)}, \text{ for every $\xi\in\mathcal{M}^*_0$}.
	\end{equation*}
	
	Note that 
	\begin{enumerate}[label=(\Alph*)]
	\item $\mathcal{M}^R_0=\{\xi\in H^\frac{1}{2}(\partial B_1):\int_{\partial B_1}\xi d\mathcal{H}^{N-1}=0\}$;
	\item $\mathcal{M}_0=\Bigr\{\xi\in H^\frac{1}{2}(\partial B_1):\int_{\partial B_1}\xi d\mathcal{H}^{N-1}=\int_{\partial B_1}x_i\xi d\mathcal{H}^{N-1}=0, i = 1,2,\ldots,N\Bigl\}$.
	\end{enumerate}
	
	We recall that
			\begin{enumerate}[label=(\Alph*)] 
			\item
			\begin{equation*}
				\partial^2 \Capa_R(B_1)[\varphi,\varphi]:=2\frac{(N-2)^2}{1-R^{-(N-2)}}\left(\int_{B_R\backslash B_1}{\vert\nabla H_R(\varphi)\vert^2}dx
				-(N-1)\int_{\partial B_1}{\varphi^2}d\mathcal{H}^{N-1}\right);
			\end{equation*}
			\item
			\begin{equation*}
				\partial^2 \Capa(B_1)[\varphi,\varphi]:=2(N-2)^2\left(\int_{B_1^c}{\vert\nabla H(\varphi)\vert^2}dx
				-(N-1)\int_{\partial B_1}{\varphi^2}d\mathcal{H}^{N-1}\right).
			\end{equation*}
			\end{enumerate}
				
	We consider first the case of relative capacity. We need to estimate the quotient $$\frac{\int_{B_R\backslash B_1}{\vert\nabla H_R(\xi)\vert^2}dx}{\int_{\partial B_1}{\xi^2}d\mathcal{H}^{N-1}}$$ from below
	for $\xi\in\mathcal{M}_0\backslash\{0\}$.	We note that it is the Rayleigh quotient for the operator $\xi\mapsto\nabla H_R(\xi)\cdot\nu$.
	Thus, we need to calculate its eigenvalues. We use spherical functions as a basis of $L^2(\partial B_1)$: $\xi=\sum_{m,n}{a_{m,n}Y_{m,n}}$.
	We now show that $H(Y_{m,n})$ can be written as $R_{m,n}(r)Y_{m,n}(\omega)$ for a suitable function \(R_{m,n}(r)\).
	Indeed, by the equation defining $H(Y_{m.n})$ we have check that 		
	\begin{equation*}
		\begin{cases}
			\Delta (R_{m,n}(r)Y_{m,n}(\omega))=0\text{ in }B_R\backslash B_1 \\
			R_{m,n}(1)Y_{m,n}=Y_{m,n} \\
			R_{m,n}(R)Y_{m,n}=0
		\end{cases}
	\end{equation*}
	Since $\tilde{\Delta}Y_{m,n}=-m(m+N-2)Y_{m,n}$, where $\tilde{\Delta}$ is the Laplace-Beltrami operator, one easily checks that 
	\begin{equation*}
		R_{m,n}(r)=-\frac{1}{R^{2m+N-2}-1}r^m+\Bigl(1+\frac{1}{R^{2m+N-2}-1}\Bigr)r^{-(N+m-2)}
	\end{equation*}
	provides  a solution. Hence, the first eigenvalue is zero and corresponds to constants, whereas the first non-zero one is $-R_{1,n}'(1)=(N-1)+\frac{1}{R^N-1}N$.
	
	For the case of the absolute  capacity we estimate the quotient $$\frac{\int_{B_1^c}{\vert\nabla H(\xi)\vert^2}dx}{\int_{\partial B_1}{\xi^2}d\mathcal{H}^{N-1}}$$ in an analogous way. The functions $R_{m,n}$ in this case
is 
\[
R_{m,n}(r)=r^{-(N+m-2)}.
\]
	The first eigenvalue is zero and corresponds to constants, the second one is $N-1$ and corresponds to the coordinate functions, the next one is $N$.
	
	\medskip
	
	\emph{Step 2: reducing to $\mathcal{M}^*_0$.} We are going to apply Step 1 to the projection $\xi_0$ of $\xi$ on $\mathcal{M}^*_0$ and show that 
	the difference $\vert\partial^2 \Capa_*(B_1)[\xi,\xi]-\partial^2 \Capa_*(B_1)[\xi_0,\xi_0]\vert$ is small.	Let $\xi$ be in $\mathcal{M}^*_\delta$. Define
	\begin{enumerate}[label=(\Alph*)]
	\item $$\xi_0:=\xi-\frac{1}{N\vert B_1\vert}\int_{\partial B_1}\xi d\mathcal{H}^{N-1};$$
	\item $$\xi_0:=\xi-\frac{1}{N\vert B_1\vert}\int_{\partial B_1}\xi d\mathcal{H}^{N-1}-\frac{1}{\vert B_1\vert}\sum_{i=1}^Nx_i\int_{\partial B_1}y_i\xi d\mathcal{H}^{N-1}.$$

	\end{enumerate}
	It is immediate from the definition that $\xi_0$ belongs to $\mathcal{M}^*_0$. We now  compare the norms of $\xi$ and $\xi_0$. We denote $c:=\xi-\xi_0$ and we write
	\begin{equation}\label{e:normxi}
	\begin{aligned}
		\Vert\xi_0\Vert^2_{H^\frac{1}{2}(\partial B_1)}&=\int_{\partial B_1}{(\xi-c)^2}d\mathcal{H}^{N-1}+
		\int_{B_R\backslash B_1}{\vert\nabla H(\xi-c)\vert^2}dx\\
		&=\Vert\xi\Vert^2_{H^\frac{1}{2}(\partial B_1)}-\Vert c\Vert^2_{H^\frac{1}{2}(\partial B_1)}
		-2\int_{\partial B_1}c\, (\xi-c)d\mathcal{H}^{N-1}\\
		&\quad-2\int_{B_R\backslash B_1}{\nabla H(c)\cdot(\nabla H(\xi)-\nabla H(c))}dx\\
		&=\Vert\xi\Vert^2_{H^\frac{1}{2}(\partial B_1)}-\Vert c\Vert^2_{H^\frac{1}{2}(\partial B_1)}.
	\end{aligned}	
	\end{equation}		
	Note that in the last equality we used integration by parts and the definition of $H$. Since $\xi$ belongs to $\mathcal{M}^*_{C\delta}$, we have
	\begin{equation}\label{e:normc}
	\Vert c\Vert^2_{H^\frac{1}{2}(\partial B_1)}\le C\Vert c\Vert^2_{L^2(\partial B_1)}\le C\delta^2 \|\xi\|^2	
	\end{equation}
	where we have used that since \(c=\xi-\xi_0\) belongs to an \(N+1\) dimensional space, the \(H^{1/2}\) and the \(L^2\) are equivalent. Now we apply Step 1 to $\xi_0$ to get
	\begin{equation*}
	\begin{aligned}
		\partial^2 \Capa_*(B_1)[\xi,\xi]&=\partial^2 \Capa_*(B_1)[\xi_0,\xi_0]+2\partial^2 \Capa_*(B_1)[\xi,c]-\partial^2 \Capa_*(B_1)[c,c]\\
		&\geq c\Vert\xi_0\Vert^2_{H^\frac{1}{2}(\partial B_1)}-2\Vert c\Vert_{H^\frac{1}{2}(\partial B_1)}\Vert \xi\Vert_{H^\frac{1}{2}(\partial B_1)}-\Vert c\Vert^2_{H^\frac{1}{2}(\partial B_1)}
	\end{aligned}	
	\end{equation*}
	and thus , by \eqref{e:normxi} and \eqref{e:normc},
	$$\partial^2 \Capa_*(B_1)[\xi,\xi]\geq c\Vert\xi_0\Vert^2_{H^\frac{1}{2}(\partial B_1)}-C\delta \|\xi\|^2\ge \frac{c}{2}\|\xi\|^2_{H^\frac{1}{2}(\partial B_1)}$$			
provided \(\delta\) is chosen sufficiently small.

\end{proof}

	\section{Stability  for bounded sets with small asymmetry}\label{penpb}
	
	This section is dedicated to the proof of the following theorem.	
	
	\begin{theorem} \label{mainthmbdd}
		There exist constants $c=c(N,R)$, $\epsilon_0=\epsilon_0(N,R)$ such that
		for any open set $\Omega\subset B_R$ with $\vert\Omega\vert=\vert B_1\vert$ and $\alpha_*(\Omega)\leq\epsilon_0$
		the following inequality holds:
		\begin{equation*}
			\Capa_*(\Omega)-\Capa_*(B_1)\geq c\alpha_*(\Omega).
		\end{equation*}
	\end{theorem}
	
	We want to reduce our problem to nearly spherical sets. To do that we argue by contradiction. Assume that there exists a sequence of domains $\tilde{\Omega}_j$ such that
	\begin{equation} \label{contrseq}
		\vert\tilde{\Omega}_j\vert=\vert B_1\vert,\text{ }
		\alpha_*(\tilde{\Omega}_j)=\epsilon_j\rightarrow 0,\text{ }
		\Capa_*(\tilde{\Omega}_j)-\Capa_*(B_1)\leq\sigma^4\epsilon_j 	
	\end{equation}
	for some $\sigma$ small enough to be chosen later. We then prove the existence of a new contradicting sequence made of smooth sets via a selection principle.
	\begin{theorem}[Selection Principle]\label{SelPr}
		There exists $\tilde{\sigma}=\tilde{\sigma}(N,R)$ such that 
		if one has a contradicting sequence $\tilde{\Omega}_j$ as 
		the one described above in \eqref{contrseq} with $\sigma<\tilde{\sigma}$, 
	          then there exists a sequence of smooth open sets $U_j$ such that
		\begin{enumerate}[label=\textup{(\roman*)}]
			\item $\vert U_j\vert=\vert B_1\vert$,
			\item $\partial U_j\rightarrow \partial B_1$ in $C^k$ for every $k$,
			\item $\limsup_{j\rightarrow\infty}\frac{\Capa_*(U_j)-\Capa_*(B_1)}{\alpha_*(\Omega_j)}\leq C\sigma$ for some $C=C(N,R)$ constant,
			\item for the case of the capacity in $\mathds{R}^N$ the barycenter of every $\Omega_j$ is in the origin.
		\end{enumerate}
	\end{theorem}

	\begin{proof} [Proof of Theorem \ref{mainthmbdd} assuming Selection Principle]
		Suppose Theorem \ref{mainthmbdd} does not hold. Then for any $\sigma>0$ we can find a contradicting
		sequence $\tilde{\Omega}_j$ as in \eqref{contrseq}. We apply Selection Principle to $\tilde{\Omega}_j$
		to get a smooth contradicting sequence $U_j$.
		
		By the properties of $\Omega_j$, we have that for $j$ big enough $U_j$ is a nearly spherical set.
		Thus, we can use Theorem \ref{mainthmnrlysphr} and get
		$$c(N,R)\leq\limsup_{j\rightarrow\infty}\frac{\Capa_*(U_j)-\Capa_*(B_1)}{\alpha_*(\Omega_j)}\leq C(N,R)\sigma.$$
		But this cannot happen for $\sigma$ small enough depending only on \(N\) and \(R\)
	\end{proof}
	
	The proof of Theorem \ref{SelPr} is based on constructing the new sequence of sets by solving a variational problem. The existence of this new sequence is established in the next section while its regularity properties are studied in Section \ref{reg}. 		
	
\section{Proof of Theorem \ref{SelPr}: Existence and first properties}	\label{s:ex}
	
\subsection{Getting rid of the volume constraint}
The first step consists in getting rid of the volume constraint in the isocapacitary inequality. Note that this has to be done locally  since, by scaling,  globally there exists no Lagrange multiplier. Furthermore, to apply the regularity theory for free boundary problems, it is crucial to introduce a \emph{monotone} dependence on the volume. To this end, let us set
	\begin{equation*}
		f_\eta(s):=
		\begin{cases}
			-\frac{1}{\eta}(s-\omega_N), \qquad &s\leq\omega_N\\
			-\eta(s-\omega_N), &s\geq\omega_N
		\end{cases}
	\end{equation*} 
	and let us consider the new functional	
	\begin{equation*}
		\mathscr{C}^*_\eta(\Omega)=\Capa_*(\Omega)+f_\eta(\vert\Omega\vert).
	\end{equation*}
We now show that the above functional is uniquely minimized by balls. Note also that \(f_\eta\) satisfies
\begin{equation}
\label{e:feta}
\eta(t-s)\le f_\eta(s)-f_\eta(t)\le \frac{(t-s)}{\eta}\qquad\text{for all \(0\le s\le t\).}
\end{equation}

\begin{lemma}[Relative capacity] \label{propofCeta'}
		There exists an $\hat{\eta}=\hat{\eta}(R)>0$ such that the only minimizer of 
		$\mathscr{C}^R_{\hat{\eta}}$ in the class of sets contained in $B_{R}$
		is  $B_1$, the unit ball centered at the origin.
		
		Moreover, there exists $c=c(R)>0$ such that for any ball $B_r$ with $0<r<R$, one has
		\begin{equation}\label{e:quantitiveballs1}
			\mathscr{C}^R_{\hat{\eta}}(B_r)-\mathscr{C}^R_{\hat{\eta}}(B_1)\geq c\vert r-1\vert.
		\end{equation}
	\end{lemma}

	\begin{lemma}[Absolute capacity] \label{propofCeta}
		There exists an $\hat{\eta}=\hat{\eta}(R)>0$ such that the only minimizer of 
		$\mathscr{C}_{\hat{\eta}}$ in the class of sets contained in $B_{R}$
		is  a translate of the unit  unit ball $B_1$.
		
		Moreover, there exists $c=c(R)>0$ such that for any ball $B_r$ with $0<r<R$, one has
		\begin{equation}\label{e:quantitiveballs2}
			\mathscr{C}_{\hat{\eta}}(B_r)-\mathscr{C}_{\hat{\eta}}(B_1)\geq c\vert r-1\vert.
		\end{equation}
	\end{lemma}

%
	\begin{proof}[Proof of Lemma \ref{propofCeta'}]
		First of all, using symmetrization we get that any minimizer of $\mathscr{C}^R_\eta$
		is a ball centered at zero. Thus, it is enough to show that for some $\eta>0$
		\begin{equation*}
			g(r):=\mathscr{C}^R_\eta(B_r)
		\end{equation*}
		attains its only minimum at $r=1$ on the interval $(0,R)$. We  recall that the (relative) capacitary potential of \(B_r\) in \(B_R\) is given by
		\[
		u_R=\min\Biggl\{\frac{\bigl(|x|^{-(n-2)}-R^{-(n-2)}\bigr)_+}{r^{-(n-2)}-R^{-(n-2)}},1\Biggr\}
		\]
		and thus 
		\[
		\Capa_R(B_r)=\frac{(n-2)}{r^{-(n-2)}-R^{-(n-2)}},
		\]
hence
		\begin{equation*}
			g(r)=\Capa_R(B_r)+f_\eta(\omega_N r^N)=\frac{R^{N-2}-1}{(\frac{R}{r})^{N-2}-1}\Capa_R(B_1)+f_\eta(\omega_N r^N).
		\end{equation*}
		For convenience let us denote  
		\begin{equation*}
			\varphi(r):=\Capa_R(B_r)=c_1(R)\frac{R^{N-2}-1}{R^{N-2}-r^{N-2}}r^{N-2},
		\end{equation*}	
		and note that 
		\begin{equation*}
			\varphi'(r)=c_1(R)(N-2)\left(\frac{R^{N-2}-1}{R^{N-2}-r^{N-2}}r^{N-3}+\frac{R^{N-2}-1}{(R^{N-2}-r^{N-2})^2}r^{N-3}r^{N-2}\right).
		\end{equation*}	
		Now we consider separately the two cases $0<r\leq 1$ and $1\leq r\leq R$.\\
		\begin{itemize}
			\item $0<r\leq 1$
			\begin{equation*}
				g'(r)=c_2(R)\left(\frac{R^{N-2}-1}{R^{N-2}-r^{N-2}}r^{N-3}+\frac{R^{N-2}-1}{(R^{N-2}-r^{N-2})^2}r^{N-3}r^{N-2}\right)-\frac{1}{\eta}\omega_N Nr^{N-1}.
			\end{equation*}
			For $r\in (1/2,1)$
			\begin{equation*}
				g'(r)\leq c(R)-\frac{1}{\eta}\omega_N N\left(\frac{1}{2}\right)^{N-1}.
			\end{equation*}
			If we take $\eta<\eta(R)\ll 1$, then	$g'(r)<-c_3(R)$ for $r\in (\frac{1}{2},1)$ 
			and thus $g(r)$ attains its minimum at $r=1$ on that interval.
			
			Moreover  for \(r\in (0,1/2)\) 
			\begin{equation*}
				g(r)=\frac{R^{N-2}-1}{(\frac{R}{r})^{N-2}-1}\Capa_R(B_1)+\frac{1}{\eta}\left(\omega_N\left(1-r^N\right)\right)\geq
				\frac{1}{\eta}\left(\omega_N\left(1-\left(\frac{1}{2}\right)^N\right)\right).
			\end{equation*}
			Since \(g(1)=\Capa_R(B_1)=c(R)\) we can take \(\eta\) small enough depending only on \(R\) to ensure that \(g(r)\ge g(1)\) for all \(r\in [0,1/2)\).  
			
			\item $1\leq r<R$
			\begin{equation*}
				\begin{aligned}
					g'(r)&=c(R)\left(\frac{R^{N-2}-1}{R^{N-2}-r^{N-2}}r^{N-3}+\frac{R^{N-2}-1}{(R^{N-2}-r^{N-2})^2}r^{N-3}r^{N-2}\right)-\eta\omega_N Nr^{N-1}\\
					&\geq c(R)-\eta\omega_N NR^{N-1}.\\
				\end{aligned}	
			\end{equation*}	
			Taking  $\eta\ll 1$ depending only on \(r\) we get $g'(r)>c_4(R)$ for $r\in (1,R)$ 
			and thus $g(r)$ attains its minimum at $r=1$ also on this interval.\\
		\end{itemize}
		To prove the last claim just note that 
		\begin{equation*}
		  \lim_{r\rightarrow 1^-}g'(r)\leq-c_3\qquad \lim_{r\rightarrow 1^+}g'(r)\geq c_4. 
		\end{equation*}
	\end{proof}
	
	\begin{proof}[Proof of Lemma \ref{propofCeta}]
		The proof works exactly as the one in the previous lemma, just using the equality
		\[
		\Capa(B_r)=\Capa(B_1)r^{n-2}.
		\]
	\end{proof}	

	\subsection{A penalized minimum problem}
	The sequence in Theorem \ref{SelPr} is obtained by solving the following minimum problem.
	\begin{equation} \label{pertpb}
		\min{\{\mathscr{C}^*_{\hat{\eta},j}(\Omega):\Omega\subset B_R\}},
	\end{equation}
	where 
	\[
	\mathscr{C}^*_{\hat{\eta},j}(\Omega)=\mathscr{C}^*_{\hat{\eta}}(\Omega)+\sqrt{\epsilon_j^2+\sigma^2(\alpha_*(\Omega)-\epsilon_j)^2}
	=\Capa_*(\Omega)+f_{\hat{\eta}}(\vert\Omega\vert)+\sqrt{\epsilon_j^2+\sigma^2(\alpha_*(\Omega)-\epsilon_j)^2}.
	\]
	
	We start proving the existence of minimizers. As in \cite{fkstab}, in order to ensure the continuity of the asymmetry term, one needs   to construct a minimizing sequence with equibounded perimeter. Recall also that a set is said to be quasi open if it is the zero level set of a \(W^{1,2}\) function.
	\begin{lemma}
		There exists $\sigma_0=\sigma_0(N,R)>0$ such that for every $\sigma<\sigma_0$
		the minimum in \eqref{pertpb} is attained by a quasi-open set $\Omega_j^*$.
		Moreover, perimeters of $\Omega_j^*$ are bounded independently on $j$.
	\end{lemma}
	\begin{proof}	We will focus on the capacity with respect to the ball. For the case of capacity in $\mathds{R}^n$ 
		one simply replaces  $W_0^{1,2}(B_R)$ by \(D^{1,2}(\mathds R^N)\). 
		
		\medskip
		\noindent
		\textbf{Step 1: finding minimizing sequence with  bounded perimters.}		
		We consider $\{V_k\}_{k\in\mathds{N}}$ -- a minimizing sequence for $\mathscr{C}^R_{\hat{\eta},j}$, satisfying
		\begin{equation*}
			\mathscr{C}^R_{\hat{\eta},j}(V_k)\leq \inf{\mathscr{C}^R_{\hat{\eta},j}}+\frac{1}{k}.
		\end{equation*}
		We denote by $v_k$ the capacitary potentials of $V_k$, so $V_k=\{x\in B_R: v_k=1\}$.
		We take as a variation the slightly enlarged set $\tilde{V}_k$:
		\begin{equation*}
			\tilde{V}_k=\{x\in B_R: v_k>1-t_k\},
		\end{equation*}
		where $t_k=\frac{1}{\sqrt{k}}$.
		
		Note that the function $\tilde{v}_k=\frac{\min(v_k, 1-t_k)}{1-t_k}$ is in $W^{1,2}_0(B_R)$ and  \(v_k=1\) on $\tilde{V}_k$, so we can bound the capacity of $\tilde{V}_k$ by $\cp{B_R}{\tilde{v}_k}$. Since $V_k$ is almost minimizing, we write
		\begin{equation*}
			\begin{aligned}
				&\int_{\{v_k<1\}}{\vert\nabla v_k\vert^2}dx+f_{\hat{\eta}}(|\{v_k=1\}|)+\sqrt{\epsilon_j^2+\sigma^2(\alpha(\{v_k=1\})-\epsilon_j)^2}\\
				&\leq\int_{\{v_k<1-t_k\}}{\left\vert\nabla \left(\frac{v_k}{1-t_k}\right)\right\vert^2}dx+f_{\hat{\eta}}(|\{v_k\geq 1-t_k\}|)+\sqrt{\epsilon_j^2+\sigma^2(\alpha(\{v_k\geq 1-t_k\})-\epsilon_j)^2}+\frac{1}{k}.
			\end{aligned}	
		\end{equation*}
		We use \eqref{e:feta} and the fact that the function $t\mapsto\sqrt{\epsilon_j^2+\sigma^2(t-\epsilon_j)^2}$ is \(1\) Lipschitz to get
		\begin{equation*}
			\begin{aligned}
				&\int_{\{1-t_k<v_k<1\}}{\vert\nabla v_k\vert^2}dx+\hat\eta|\{1-t_k<v_k<1\}|\\
				&\leq\sigma\big(\vert\alpha(\{v_k\geq 1-t_k\}-\alpha(\{v_k=1\})\vert\big)+\frac{1}{k}+\int_{\{v_k<1-t_k\}}{\left(\left(\frac{1}{1-t_k}\right)^2-1\right)\vert\nabla v_k\vert^2}dx\\
				&\leq	C(R)\sigma\vert\{1-t_k<v_k\leq 1\}\vert+\frac{1}{k}+\left(\left(\frac{1}{1-t_k}\right)^2-1\right) \Capa_R(V_k)\\
				&\leq	C(R)\sigma\vert\{1-t_k<v_k\leq 1\}\vert+\frac{1}{k}+c(N,R)t_k,
			\end{aligned}	
		\end{equation*}
		where in the second inequality we used Lemma \ref{propasym}, \ref{asymlip}. Taking $\sigma<\frac{\hat\eta}{2C(R)}$, we obtain
		\begin{equation*}
			\int_{\{1-t_k<v_k<1\}}{\vert\nabla v_k\vert^2}dx+\frac{\hat \eta}{2}(|\{1-t_k<v_k<1\}|)\leq\frac{1}{k}+c(N,R)t_k.
		\end{equation*}
		We estimate the left-hand side from below, using the arithmetic-geometric  mean inequality, the Cauchy-Schwarz inequality, and the co-area formula.
		\begin{equation*}
			\begin{aligned}
				&\int_{\{1-t_k<v_k<1\}}{\vert\nabla v_k\vert^2}dx+\frac{\hat\eta}{2}(|\{1-t_k<v_k<1\}|)\\
				&\geq 2\left(\int_{1-t_k<v_k<1}{\vert\nabla v_k\vert^2}dx\right)^\frac{1}{2}\left(\frac{\hat \eta}{2}(|\{1-t_k<v_k<1\}|)\right)^\frac{1}{2}\\
				&\geq \sqrt{2\hat\eta}\int_{1-t_k<v_k<1}{\vert\nabla v_k\vert}dx
				=\sqrt{2\hat\eta}\int_{1-t_k}^1{P(v_k>s)}ds.
			\end{aligned}	
		\end{equation*}
where \(P(E)\) denotes the De Giorgi perimeter of a set \(E\). Hence, there exists a level $1-t_k<s_k<1$ such that for $\hat{V}_k=\{v_k>s_k\}$
		\begin{equation*}
			P(\hat{V}_k)\leq\frac{1}{t_k}\int_{1-t_k}^1{P(\{v_k>s\})}ds\leq\frac{1}{t_k\sqrt{2\hat\eta}k}+c(N,R)=\frac{1}{\sqrt{2\hat\eta k}}+c(N,R).
		\end{equation*}	
		where in the last equality we have used that \(t_k=\frac{1}{\sqrt{k}}\).  
		These $\hat{V}_k$ will give us the desired "good" minimizing sequence, indeed 
		\begin{equation*}
			\begin{aligned}
				\mathscr{C}^R_{\hat{\eta},j}&(\hat{V}_k)
				\\
				&\leq\mathscr{C}^R_{\hat{\eta},j}(V_k)+f_{\hat{\eta}}(\vert\{ v_k>s_k\}\vert)-f_{\hat{\eta}}(\vert \{v_k=1\}\vert)
				+C\sigma\vert\{1-s_k<v_k<1\}\vert \leq\mathscr{C}^R_{\hat{\eta},j}(V_k),
			\end{aligned}
		\end{equation*}
		where in the first inequality we have used that \(\hat V_k\subset V_k\) and in the second that, thanks to our choice of \(\sigma\),
		\[
		f_{\hat{\eta}}(\vert \{v_k>s_k\}\vert)-f_{\hat{\eta}}(\vert \{v_k=1\}\vert)
				+C\sigma\vert\{1-s_k<v_k<1\}\le (C\sigma-\hat\eta)\vert\{1-s_k<v_k<1\}\le 0.
		\]
		
\medskip
\noindent
\textbf{Step 2: Existence of a minimizer.}		
		Since $\{\hat{V}_k\}_k$ is a sequence with  equibounded perimeter,s there exists a Borel set $\hat{V}_\infty$ such that up to a (not relabelled) subsequence 
		$$ 1_{\hat{V}_k}\rightarrow 1_{\hat{V}_\infty} \text{ in } L_1(B_R)\text{ and a.e. in }B_R, \qquad P(\hat{V}_\infty)\leq C(N,R).$$
We want to show that $\hat{V}_\infty$ is a minimizer for $\mathcal{C}_{\eta,j}$.	We set $\hat{v}_k=\frac{\min(v_k, s_k)}{s_k}$ and we note that they are the capacitary potentials of \(\hat V_k\). Moreover  the sequence $\{\hat{v}_k\}_k$ is bounded in $W^{1,2}_0(B_R)$. Thus, there exists a function $\hat{v}\in W^{1,2}_0(B_R)$ such that up to a (not relabelled) subsequence $$\hat{v}_k\rightarrow \hat{v} \text{ strongly in }L^2(B_R)\text{ and a.e. in }B_R.$$
		Let us define  $\hat{V}=\{x:\hat{v}=1\}$, we want to show that \(\hat V\) is a minimizer. First,  note that
		$$1_{\hat{V}}(x)\geq \limsup{1_{\hat{V}_k}}(x)=1_{\hat{V}_\infty}(x)\qquad\text{ for a.e. }x\in B_R,$$
hence \(|\hat V_\infty \setminus \hat V|=0\). Moreover, by the lower semicontinuity of Dirichlet integral, the monotonicity of \(f_{\hat{\eta}}\) and the continuity of $\alpha$ with respect to the $L^1$ convergence, we have
		\begin{equation}
		\begin{aligned} \label{Vinftyismin}
		\inf\mathscr{C}^R_{\hat{\eta},j}&=\lim_k \int \vert\nabla\hat{v}_k\vert^2+f_{\hat{\eta}}(\vert\hat{V}_k\vert)+\sqrt{\epsilon_j^2+\sigma^2(\alpha(\hat{V}_k)-\epsilon_j)^2}
			\\
			&\ge \Capa_R(\hat{V})+f_{\hat{\eta}}(\vert\hat{V}_\infty\vert)
			+\sqrt{\epsilon_j^2+\sigma^2(\alpha(\hat{V}_\infty)-\epsilon_j)^2}.
			&\ge \Capa_R(\hat{V})+f_{\hat{\eta}}(\vert\hat{V}\vert)
		\end{aligned}	
		\end{equation}
		Hence
		\begin{equation*}
		\begin{aligned} 
			\Capa_R(\hat{V})&+f_{\hat{\eta}}(\vert\hat{V}_\infty\vert)
			+\sqrt{\epsilon_j^2+\sigma^2(\alpha(\hat{V}_\infty)-\epsilon_j)^2}
			\leq\inf\mathscr{C}^R_{\hat{\eta},j}(\Omega)\\
			&\leq \Capa_R(\hat{V})+f_{\hat{\eta}}(\vert\hat{V}\vert)
			+\sqrt{\epsilon_j^2+\sigma^2(\alpha(\hat{V})-\epsilon_j)^2}.
		\end{aligned}	
		\end{equation*}
		Using Lemma \ref{propasym} \ref{asymlip}   we get 
		$$f_{\hat{\eta}}(\vert\hat{V}_\infty\vert)-f_{\hat{\eta}}(\vert\hat{V}\vert)\leq C\sigma\vert\hat{V}\Delta\hat{V}_\infty\vert=C\sigma\vert\hat{V}\setminus \hat{V}_\infty\vert.$$
		Since $\vert\hat{V}\vert\geq\vert\hat{V}_\infty\vert$, \eqref{e:feta} and  our choice of \(\sigma\)   yield
		$$\hat \eta\vert\hat{V}\backslash\hat{V}_\infty\vert\leq f_{\hat{\eta}}(\vert\hat{V}_\infty\vert)-f_{\hat{\eta}}(\vert\hat{V}\vert)\leq C\sigma\vert\hat{V}\backslash\hat{V}_\infty\vert\le \frac{\hat{\eta}}{2}\vert\hat{V}\backslash\hat{V}_\infty\vert,$$
	from which we conclude that \(|\hat V\Delta \hat V_\infty|=0\) and thus, by \eqref{Vinftyismin} that \(\hat V\) is the desired minimizer.
	\end{proof}
	
	\subsection{First properties of the minimizers}
	
	Let us conclude by establishing some properties of the minimizers of \eqref{pertpb}.
	\begin{lemma} \label{firstpropmin}
		Let $\{\Omega_j\}$ be a sequence of minimizers for \eqref{pertpb}.
		Then the following properties hold:
		\begin{enumerate}[label=\textup{(\roman*)}]
			\item 
			$\vert\alpha_*(\Omega_j)-\epsilon_j\vert\leq 3\sigma\epsilon_j$;
			\item 
			$\big\vert\vert\Omega_j\vert-\vert B_1\vert\big\vert\leq C\sigma^4\epsilon_j$;
			\item
				\begin{enumerate}[label=(\Alph*)]
					\item for the capacity in $\mathds{R}^n$
						 up to translations $\Omega_j\rightarrow B_1$ in $L^1$,
					\item for the relative capacity	 
						 $\Omega_j\rightarrow B_1$ in $L^1$;
				\end{enumerate}
			\item 
			$0\leq\mathscr{C}^*_{\hat{\eta}}(\Omega_j)-\mathscr{C}^*_{\hat{\eta}}(B_1)\leq\sigma^4\epsilon_j$.
		\end{enumerate}
	\end{lemma}
	\begin{proof}
		Recall that the  sequence $\{\Omega_j\}$ was obtained by a  sequence $\{\tilde{\Omega}_j\}$ satisying
		\begin{enumerate}
			\item $\vert\tilde{\Omega}_j\vert=\vert B_1\vert$,
			\item $\alpha_*(\tilde{\Omega}_j)=\epsilon_j$,
			\item $\Capa_*(\tilde{\Omega}_j)-\Capa_*(B_1)\leq\sigma^4\epsilon_j$.
		\end{enumerate}
		We now use $\{\tilde{\Omega}_j\}$ as comparison domains for the functionals $\mathscr{C}^*_{\hat{\eta},j}$ to get
		\begin{equation} \label{comporigcontr}
			\mathscr{C}^*_{\hat{\eta}}(\Omega_j)+\epsilon_j\leq\mathscr{C}^*_{\hat{\eta},j}(\Omega_j)
			\leq\mathscr{C}^*_{\hat{\eta},j}(\tilde{\Omega_j})=\mathscr{C}^*_{\hat{\eta}}(\tilde{\Omega_j})+\epsilon_j
			\leq\mathscr{C}^*_{\hat{\eta}}(B_1)+\epsilon_j(1+\sigma^4),
		\end{equation}
		implying that 
		$$\mathscr{C}^*_{\hat{\eta}}(\Omega_j)-\mathscr{C}^*_{\hat{\eta}}(B_1)\leq\epsilon_j\sigma^4,$$
		which proves (iv). Note that we defined $f_{\hat{\eta}}$ in such a way that $\mathscr{C}^*_{\hat{\eta}}(\Omega_j)\geq\mathscr{C}^*_{\hat{\eta}}(B_1)$. Thus, using \eqref{comporigcontr} we also deduce that
		$$\sqrt{\epsilon_j^2+\sigma^2(\alpha_*(\Omega_j)-\epsilon_j)^2}\leq\epsilon_j(1+\sigma^4),$$
		which gives (i).
		To estimate the volume of $\Omega_j$, we use the classical isocapacitary inequality and properties of $f_{\hat{\eta}}$ and \eqref{e:quantitiveballs1}, \eqref{e:quantitiveballs2}. Indeed, let $B^j$ be the ball centered in the origin such that $\vert B^j\vert=\vert \Omega_j\vert$.Then 
		$$\sigma^4\epsilon_j\geq\mathscr{C}^*_{\hat{\eta}}(\Omega_j)-\mathscr{C}^*_{\hat{\eta}}(B_1)
		\geq\mathscr{C}^*_{\hat{\eta}}(B^j)-\mathscr{C}^*_{\hat{\eta}}(B_1)\geq c(R)\big|\vert\Omega_j\vert -\vert B_1\vert\big|,$$
		where in the last inequality we have used \eqref{e:quantitiveballs1}, \eqref{e:quantitiveballs2}. This proves (ii). To prove (ii) we recall that the sets  $\Omega_j$ have equibounded perimeter. Hence, the sequence  $\{\Omega_j\}_j$ is precompact in $L^1(B_R)$. Since the asymmetry is continuous with respect to $L^1$ convergence  any limit set has zero asymmetry. The only set with zero asymmetry is the unit ball (or a translated unit ball in the case of the absolute), proving (iii).
	\end{proof} 

	\section{Proof of Theorem \ref{SelPr}: Regularity}\label{reg}
	
In this section, we show that  the sequence of minimizers of  \eqref{penpb} converges smoothly to  the unit ball. This will be done by relying on the regularity theory for free boundary problems established in \cite{AC}.
	
	\subsection{Linear growth away from the free boundary} Let $u_j$ be the capacitary potential for $\Omega_j$, a minimizer of \eqref{pertpb}.	Let us  $v_j:=1-u_j$, so that $\Omega_j=\{v_j=0\}$, $v_j=1$ on $\partial B_R$, following \cite{AC} we are going to show that 
	\[
	v_j(x)\sim \dist(x,\Omega_j).
	\]
where the implicit constant depends only on \(R\). The above estimate is obtained by suitable comparison estimates. In order to be able to perform them with constants which depend only on \(R\), we need to know that \(\{u_j=1\}\) is uniformly far from \(\partial B_R\). This will be achieved by first establishing  (uniform in \(j\)) H\"older  continuity of  \(u_j\).
	
	\subsubsection{H\"older continuity} 
	The proof of H\"older continuity is quite standard and it is based  on establishing a decay estimate for the integral oscillation of \(u_j\). Since, thanks to the minimizing property, $u_j$ is close to the harmonic function in $B_r(x_0)\cap B_R$ with the same boundary value, we  start by recalling  the decay of the harmonic functions both in the interior and at the boundary. The following is well known, see for instance  \cite[Proposition 5.8]{GM}.
	
	\begin{lemma} \label{harmgrowth}
		Suppose $w\in W^{1,2}(\Omega)$ is harmonic, $x_0\in\Omega$.
		Then there exists a constant $c=c(n)$ such that for any balls $B_{r_1}(x_0)\subset B_{r_2}(x_0)\Subset\Omega$	
		\begin{equation}
			\fint_{B_{r_1}(x_0)}\left(w-\fint_{B_{r_1}(x_0)}w\right)^2\leq c\left(\frac{r_1}{r_2}\right)^2\fint_{B_{r_2}(x_0)}\left(w-\fint_{B_{r_2}(x_0)}w\right)^2.
		\end{equation}
	\end{lemma}
	
Next lemma studies the decay at the boundary, the result is well known. Since we have not been able to find a precise reference for this statement, we report its simple proof.
	\begin{lemma}\label{p:holder}
		Let \(\Omega\) be an open set such that \(0\in \partial \Omega\) and let  $w\in W^{1,2}(B_r)$ be harmonic in $\Omega\cap B_r$, $w\equiv 0$ on  $B_r\setminus \Omega$. Assume that there exists \(\delta>0\) such that 
for  $\rho\leq r$ 
		\begin{equation*}
			\frac{\vert\Omega^c\cap B_\rho\vert}{\vert B_\rho\vert}\geq\delta.
		\end{equation*}
		Then there exist a constant $c=c(\delta)$ and an exponent $\gamma=\gamma (\delta)>0$ such that for any $0<r_1<r_2<r$ we have
		\begin{equation*}
			\fint_{B_{r_1}}w^2\leq c\left(\frac{r_1}{r_2}\right)^\gamma\fint_{B_{r_2}}w^2.
		\end{equation*}
	\end{lemma}
	\begin{rem}
	  Note that as $w$ is harmonic in \(\Omega\cap B_r\) and \(0\) on \(B_r\setminus \Omega\), $w^2$ is subharmonic in $B_r$, thus its means over balls increase
		with the radius. In particular,
		\begin{equation}
		\label{e:mv}
		\sup_{B_{r}}w^2\le c(N) \fint_{B_{3r}} w^2.
		\end{equation}
	\end{rem}
	\begin{proof}[Proof of Lemma \ref{p:holder}]
		For convenience, we assume that $r>1$ (we can reduce to this case by scaling).
		First, we note that it is enough to show the result for radii with the ratio equal to a power of $\frac{1}{4}$.
		Indeed, take $k\in\mathds{Z}_+$ such that $\frac{1}{4^{k+1}}\leq\frac{r_1}{r_2}<\frac{1}{4^k}$.
		Then
		\begin{equation*}
			\fint_{B_{r_1}}w^2\leq C 4^{-\gamma k} \fint_{B_{r_1 4^k}}w^2
			\leq C4^{-\gamma k} \fint_{B_{r_2}}w^2 \leq C4^\gamma \left(\frac{r_1}{r_2}\right)^{-\gamma} \fint_{B_{r_2}}w^2.
		\end{equation*}
		We work with powers of $\frac{1}{4}$. We start by showing 
		\begin{equation} \label{e:nome}
			\sup_{B_{\frac{1}{4}}}w\leq(1-c)\sup_{B_1}w.
		\end{equation}
		For any $\epsilon>0$ there exists some $x_0\in B_\frac{1}{4}$ such that
		$\sup_{B_{\frac{1}{4}}}w\leq w(x_0)+\epsilon$, so we can write
		\begin{equation*}
			\begin{aligned}
				&\sup_{B_{\frac{1}{4}}}w-\epsilon\leq w(x_0)\leq\fint_{B_{\frac{3}{4}}(x_0)}w
				\leq\frac{\vert\Omega\cap B_{\frac{3}{4}}(x_0)\vert}{\vert B_{\frac{3}{4}}(x_0)\vert}\sup_{B_1}w
				\\
				&=\left(1-\frac{\vert\Omega^c\cap B_{\frac{3}{4}}(x_0)\vert}{\vert B_{\frac{3}{4}}(x_0)\vert}
				\right)
				\sup_{B_1}w
				\leq\left(1-\frac{\vert\Omega^c\cap B_{\frac{1}{4}}\vert}{\vert B_{\frac{3}{4}}\vert}\right)\sup_{B_1}w\\
				&\leq\left(1-\delta\frac{\vert B_{\frac{1}{4}}\vert}{\vert B_{\frac{3}{4}}\vert}\right)\sup_{B_1}w,
			\end{aligned}
		\end{equation*}
		which proves \eqref{e:nome} since $\epsilon$ is arbitrary.
		Using induction and scaling we can extend this result to all powers of $\frac{1}{4}$.
		Indeed  $\tilde{w}(x)=w(x/4)$  satisfies the hypothesis of the theorem. Hence,
		\begin{equation*}
			\sup_{B_{\frac{1}{16}}}w=\sup_{B_{\frac{1}{4}}}\tilde{w}
			\leq (1-c)\sup_{B_1}\tilde{w}=(1-c)\sup_{B_{\frac{1}{4}}}w
			\leq (1-c)^2\sup_{B_1}w,
		\end{equation*}
and thus
		\begin{equation*} 
			\sup_{B_{\frac{1}{4^k}}}w\leq(1-c)^k\sup_{B_1}w.
		\end{equation*}
In the same way
		\begin{equation*} 
			\sup_{B_{\frac{1}{4^k}r}}w\leq(1-c)^k\sup_{B_r}w.
		\end{equation*}
		Now  
		\begin{equation*}
			\begin{aligned}
				\fint_{B_{\frac{1}{4^k}}}w^2&\leq(\sup_{B_{\frac{1}{4^k}}}w)^2
				\leq(1-c)^{2(k-1)}(\sup_{B_{\frac{1}{4}}}w)^2\\
				&\leq(1-c)^{2(k-1)}\left(\fint_{B_{\frac{3}{4}}(x_0)}w^2\right)
				\leq(1-c)^{2(k-1)}\left(c'\fint_{B_1}w^2\right),
			\end{aligned}
		\end{equation*}
	where we have used \eqref{e:mv}. We get from powers of $\frac{1}{4}$ to other radii again by scaling. This concludes the proof with  $\gamma=-\log_4(1-c)$.
	\end{proof}
	\begin{cor}\label{growth_of_means_for_harmonic}
	  Let $w$ as in the statement of Lemma \ref{p:holder}, then
		\begin{equation*}
			\fint_{B_{r_1}}\left(w-\fint_{B_{r_1}}w\right)^2\leq C\left(\frac{r_1}{r_2}\right)^\gamma\fint_{B_{r_2}}\left(w-\fint_{B_{r_2}}w\right)^2
		\end{equation*}
		for any $0<r_1<r_2<r$ with $C$ a  constant depending only on $\delta$.
	\end{cor}
	\begin{proof}
	The proof follows from Lemma \ref{p:holder} and the simple observation that for a function  $w$ vanishing on a fixed fraction of $B_\rho$, the \(L^2\) norm and the variance are comparable. Namely  there exists a constant $c=c(\delta)$  such that
		\begin{equation*}
			\frac{1}{c}\int_{B_\rho}\left(w-\fint_{B_{\rho}}w\right)^2\leq\int_{B_\rho}w^2\leq c\int_{B_\rho}\left(w-\fint_{B_{\rho}}w\right)^2.
		\end{equation*}
		Indeed, the first inequality is true for every $w$ with $c=1$. For the second one note that 
		\begin{equation*}
			\int_{B_\rho}\left(w-\fint_{B_{\rho}}w\right)^2=\int_{B_\rho}w^2-\vert B_\rho\vert \left(\fint_{B_{\rho}}w\right)^2.
		\end{equation*}
		Hence we need to  estimate $\left(\fint_{B_{\rho}}w\right)^2$ in terms of $\int_{B_\rho}w^2$. Since  $w$ is non-zero only inside $\Omega$, using H\"older inequality, we obtain
		\[
		\left(\fint_{B_{\rho}}w\right)^2\le \left(\frac{\vert\Omega\cap B_\rho\vert}{\vert B_\rho\vert}\right)\fint_{B_{\rho}}w^2
				\leq(1-\delta)\fint_{B_{\rho}}w^2,
		\]
		hence
		\begin{equation*}
			\int_{B_\rho}\left(w-\fint_{B_{\rho}}w\right)^2\geq\delta\int_{B_\rho}w^2,
		\end{equation*}
		concluding the proof.
	\end{proof}
	
	To prove H\"{o}lder continuity of $u_j$  we will use  several times the following comparison estimates. 
	\begin{lemma}\label{varcap}Let \(u_j\) be the capacitary potential of  a minimizer for \eqref{pertpb}. Let $A\subset B_R$ be an open set with Lipschitz  boundary and  let $w\in W^{1,2}(\mathds{R}^n)$ coincide with $u_j$ on the boundary of $A$ in the sense of traces.

		Then
		$$\int_A{\vert\nabla u_j\vert^2}dx-\int_A{\vert\nabla w\vert^2}dx\leq \left(\frac{1}{\hat \eta}+C\sigma\right)\big| A\cap \left(\{u=1\}\Delta\{w=1\}\right)\big|.$$
		Moreover, if $u_j\leq w\leq 1$ in $A$, then
		$$\int_A{\vert\nabla u_j\vert^2}dx+\frac{\hat \eta}{2}\big|A\cap \left(\{u=1\}\Delta\{w=1\}\right)\big|\leq\int_A{\vert\nabla w\vert^2}dx,$$
		provided \(\sigma\le \sigma(R)\).
	\end{lemma}
	\begin{proof}
		We prove the result for the relative capacity. The case of the capacity in $\mathds{R}^N$ 
		can be treated in the same way. Since \(u_j\) is fixed we drop the subscript \(j\). Consider $\tilde{u}$ defined as
		$$
			\begin{cases}
				\tilde{u}=w\qquad&\text{ in }A\\
				\tilde{u}=u&\text{ else.}
			\end{cases}.
		$$
		Take $\tilde{\Omega}=\{\tilde{u}=1\}$ as a comparison domain.
		Since $\Omega$ is minimizing, we can write
		\begin{equation*}
		\begin{aligned}
			\cp{B_R}{u}&+f_{\hat{\eta}}(\Omega)+\sqrt{\epsilon_j^2+\sigma^2(\alpha_R(\Omega)-\epsilon_j)^2}
			=\mathscr{C}^R_{\hat{\eta},j}(\Omega)\\
			&\leq\mathscr{C}^R_{\hat{\eta},j}(\tilde{\Omega})
			\leq\cp{B_R}{\tilde{u}}+f_{\hat{\eta}}(\tilde{\Omega})+\sqrt{\epsilon_j^2+\sigma^2(\alpha(\tilde{\Omega})-\epsilon_j)^2}.
		\end{aligned}
		\end{equation*}
Hence, by Lemma \ref{propasym}, \ref{asymlip} and \eqref{e:feta}.
		\begin{equation*}
		\begin{aligned}
			\cp{A}{u}-\cp{A}{w}
			\leq\vert f_{\hat{\eta}}(\Omega)-f_{\hat{\eta}}(\tilde{\Omega})\vert
			+C\sigma\vert\Omega\Delta\tilde{\Omega}\vert
			\leq(\frac{1}{\hat\eta}+C\sigma)\vert\Omega\Delta\tilde{\Omega}\vert.
		\end{aligned}
		\end{equation*}
		To prove the second inequality we observe that  $u\leq w\leq 1$ implies $\{u=1\}\subset\{\tilde{u}=1\}$, i.e. \(\Omega\subset \tilde \Omega\). Hence, by \eqref{e:feta}:
		\[
		 \cp{A}{u}-\cp{A}{w}
			\leq -f_{\hat{\eta}}(\Omega)+f_{\hat{\eta}}(\tilde{\Omega})
			+C\sigma\vert\Omega\Delta\tilde{\Omega}\vert
			\leq -\hat\eta|\tilde \Omega\setminus  \Omega|+C\sigma|\tilde \Omega\setminus  \Omega|,
		\]
		from which the inequality follows choosing \(\sigma\) small enough.
	\end{proof}	
	\begin{rem}\label{varcapharm}
		Note that if $w$ is harmonic in $A$, then 
		$$\cp{A}{u}-\cp{A}{w}=\cp{A}{(u-w)},$$
		meaning that the first inequality from the lemma becomes
		\begin{equation}\label{e:harmonic}
		\cp{A}{(u-w)}\leq \left(\frac{1}{\hat \eta}+C\sigma\right)\big| A\cap\{u=1\}\Delta\{w=1\}\big|.
		\end{equation}
	\end{rem}

Let us also recall the following technical result
	\begin{lemma}[ [Lemma 5.13 in \cite{GM}] \label{techlemmagrowth}
		Let $\phi:\mathds{R}^+\rightarrow\mathds{R}^+$ be a non-decreasing function satisfying
		$$\phi(\rho)\leq A\left[\left(\frac{\rho}{R}\right)^\alpha+\epsilon\right]\phi(R)+BR^\beta,$$
		for some $A,\alpha,\beta>0$, with $\alpha>\beta$ and for all $0<\rho\leq R\leq R_0$, where $R_0>0$ is given. 
		Then there exist constants $\epsilon_0=\epsilon_0(A,\alpha,\beta)$ and $c=c(A,\alpha,\beta)$ 
		such that if $\epsilon\leq\epsilon_0$, we have
		$$\phi(\rho)\leq c\left[\frac{\phi(R)}{R^\beta}+B\right]\rho^\beta$$
		for all $0\leq\rho\leq R\leq R_0$.
	\end{lemma}

	\begin{lemma} \label{holder_continuity} There exists \(\alpha\in (0,1/2)\) such that every minimizer of \eqref{pertpb} staisfies 
		$u_j\in C^{0,\alpha}(\overline{B_R})$. Moreover, the H\"older norm is bounded by a constant independent  on $j$.
	\end{lemma}
	\begin{proof}
	 Let us extend $u_j$ by $0$ outside of $B_R$. As usual, we drop the subscript $j$.
By Camapanato's criterion it is enough to show that 
\[
\phi(r):=\int_{B_r(x_0)}{\Big\vert u-\fint_{B_r(x_0)}{u}\Big\vert^2}\le Cr^{2\alpha}
\]		
for all \(r\) small enough (say less that \(1/2\)).

\noindent
\textbf{Step 1: estimates on the boundary.}
		Let $x_0\in \partial B_R$.	Let  $w$ be  the harmonic extension of \(u\) in $B_{r'}(x_0)\cap B_R$ . By Corollary \ref{growth_of_means_for_harmonic} we know that 
		\begin{equation*}
			\int_{B_r(x_0)}{\left\vert w-\fint_{B_r(x_0)}{w}\right\vert^2}\leq C\left(\frac{r}{r'}\right)^{N+\gamma}\int_{B_{r'}(x_0)}{\left\vert w-\fint_{B_{r'}(x_0)}{w}\right\vert^2}
		\end{equation*} 
		for some $\gamma>0$. Let   $g:=u-w$. Then
		\begin{equation*}
			\begin{aligned}
				\int_{B_r(x_0)}{\left\vert u-\fint_{B_r(x_0)}{u}\right\vert^2}dx
				&\leq 2\int_{B_r(x_0)}{\left\vert w(x)-\fint_{B_r(x_0)}{w}\right\vert^2}dx+2\int_{B_r(x_0)}{\left\vert g-\fint_{B_r(x_0)}{g}\right\vert^2}\\
				&\leq 2C\left(\frac{r}{r'}\right)^{N+\gamma}\int_{B_{r'}(x_0)}{\left\vert w-\fint_{B_r(x_0)}{w}\right\vert^2}dx+2\int_{B_{r'}(x_0)}{\left\vert g\right\vert^2}\\
				&\leq C\left(\frac{r}{r'}\right)^{N+\gamma}\int_{B_{r'}(x_0)}{\left\vert u-\fint_{B_r(x_0)}{u}\right\vert^2}dx+C\int_{B_{r'}(x_0)}{\vert g\vert^2}.
			\end{aligned}	
		\end{equation*}
To estimate  $\int_{B_{r'}(x_0)}{\vert g(x)\vert^2}dx$ we recall that  $g\in W_0^{1,2}(B_{r'}(x_0))$ and vanishes outside \(B_{r'}(x_0)\cap B_R\), hence by Poincar\'e's inequality and \eqref{e:harmonic}
		\begin{equation*}
			\int_{B_{r'}(x_0)}{\vert g\vert^2}\leq C(r')^2\int_{B_{r'}(x_0)\cap B_R}{\vert \nabla g\vert^2}\le C(r')^{N+2}.
		\end{equation*}
Combining the last two inequalities, we get
		\begin{equation*}
			\phi(r)\leq c\left(\frac{r}{r'}\right)^{N+\gamma}\phi(r')+C(r')^{N+2}.
		\end{equation*}
		Using Lemma \ref{techlemmagrowth} we obtain
		\begin{equation*}
			\phi(r)\leq c\left(\left(\frac{r}{r'}\right)^{N+\gamma}\phi(r')+Cr^{N+\gamma}\right)
		\end{equation*}
		for any $r<r'<1$. In particular,
		\begin{equation*}
			\phi(r)\leq c\left(\Vert u\Vert^2_{L^2(\mathds{R}^N)}+C\right)r^{N+\gamma}.
		\end{equation*}	
		
		\noindent
		\textbf{Step 2: estimates at the interior.}
		Assume that $x_0\in B_R$, $r<r'<\dist(x_0,\partial B_R)$, so that $B_r(x_0)\subset B_{r'}(x_0)\subset B_R$.
		Then one can proceed in the same way as in the previous step using Lemma  \ref{harmgrowth} 
		instead of Corollary \ref{growth_of_means_for_harmonic}.	 Hence  
		\begin{equation*}
			\phi(r)\leq C\left(\left(\frac{r}{r'}\right)^{N+\gamma}\phi(r')+Cr^{N+\gamma}\right)
		\end{equation*}
		for $r<r'<\dist(x_0,\partial B_R)$ and, in particular,
		\begin{equation}
			\phi(r)\leq c\left(\left(\frac{1}{\dist(x_0,\partial B_R)}\right)^{N+\gamma}\Vert u\Vert^2_{L^2(\mathds{R}^N)}+C\right)r^{N+\gamma}.
		\end{equation}	
		
		\noindent
		\textbf{Step 3: global estimates.} We now combine the previous steps, distinguishing several cases:
		
		\begin{itemize}
			\item $\dist(x_0, \partial B_R)>1/2$. By Step 2 
			\begin{equation*}
				\phi(r)\leq C\left(\Vert u\Vert^2_{L^2(\mathds{R}^N)}+C\right)r^{N+\gamma}.
			\end{equation*}	
		
			\item $r\le\rho:= \dist(x_0, \partial B_R)\le 1/2$.
			Let $y_0=R\frac{x_0}{\vert x_0\vert}$ be the intersection of the ray $[0,x_0)$ with $\partial B_R$. Then, using Step 2 and Step 1, we have
			\begin{equation*}
				\begin{aligned}
					\phi(r)&\leq C\left(\left(\frac{r}{\rho}\right)^{N+\gamma}\phi(\rho)+Cr^{N+\gamma}\right)
					\\
					&=C\left(\left(\frac{r}{\rho}\right)^{N+\gamma}\int_{B_{\rho}(x_0)}{\left\vert u-\fint_{B_{\rho}(x_0)}{u}\right\vert^2}+Cr^{N+\gamma}\right)\\
					&\leq C\left(\left(\frac{r}{\rho}\right)^{N+\gamma}\int_{B_{2\rho}(y_0)}{\left\vert u-\fint_{B_{2\rho}(y_0)}{u}\right\vert^2}+Cr^{N+\gamma}\right)\\
					&\leq C\left(\left(\frac{r}{\rho}\right)^{N+\gamma}(2\rho)^{N+\gamma}\int_{B_{1}(y_0)}{\left\vert u-\fint_{B_{1}(y_0)}{u}\right\vert^2}+Cr^{N+\gamma}\right)\\
					&\leq C\left(\Vert u\Vert^2_{L^2(\mathds{R}^N)}+C\right)r^{N+\gamma}.
				\end{aligned}	
			\end{equation*}	

			\item $\rho:=\dist(x_0, \partial B_R)\le r\le 1/2 $. Again we set $y_0$ to be the radial projection of $x_0$ onto $\partial B_R$.
			We use Step 1 and get
			\begin{equation*}
				\begin{aligned}
					\phi(r)&=\int_{B_{r}(x_0)}{\left\vert u-\fint_{B_{r}(x_0)}{u}\right\vert^2} 
					\leq\int_{B_{2r}(y_0)}{\left\vert u-\fint_{B_{2r}(y_0)}{u}\right\vert^2}\\
					&\leq C\left(r^{N+\gamma}\int_{B_{1}(y_0)}{\left\vert u-\fint_{B_{1}(y_0)}{u}\right\vert^2}+Cr^{N+\gamma}\right)\\
					&\leq C\left(\Vert u\Vert^2_{L^2(\mathds{R}^N)}+C\right)r^{N+\gamma}.
				\end{aligned}	
			\end{equation*}	
		\end{itemize}
		In conclusion,
		\[
		\phi(r)\le C\left(\Vert u\Vert^2_{L^2(\mathds{R}^N)}+C\right)r^{N+\gamma},
		\]
		which by Campanato criterion implies that \(u\in C^{\frac{\gamma}{2}}\). Note furthermore that  the dependence on $j$ is realized only by the $L^2$ norm of $u_j$ which is uniformly  bounded by $\sqrt{\vert B_R\vert}$.
	\end{proof}
	
	\subsubsection{Lipschitz continuity and density estimates on the boundary}
	
	We now prove two lemmas similar to those in Section 3 of \cite{AC}. These are obtained by adding or removing a small ball from  an optimizer of \eqref{pertpb}. Since our competitors are constrained to lie in \(B_R\) removing a ball is  not a problem. On the other hand adding might lead to a non admissible competitor. For the case of the relative capacity,  we use the   H\"older estimate of the previous section. Indeed it implies that there exists \(\rho_0=\rho_0(R)>0\) such that 
	\begin{equation}\label{e:inc}
	\Omega_j\subset B_{R-\rho_0}.
\end{equation}
	
	\begin{lemma}
		For $\kappa<1$ there is a constant $c=c(N,\kappa,R)$ such that if \(u_j\) is a minimizer for \eqref{pertpb} and  \(v_j=1-u_j\) satisfies 
		\begin{equation}\label{e:piccolo}
			\fint_{\partial B_r(x_0)}{v_j}\leq cr,
		\end{equation}
		then $v_j=0$ in $B_{\kappa r}(x_0)$. In the case of the relative capacity we assume \(r\le \rho_0\) where \(\rho_0\) is as in \eqref{e:inc}.
	\end{lemma}
	\begin{proof}
		We drop the subscript $j$ for simplicity. We first check that $B_{\kappa r}(x_0)\subset B_R$. By our restriction on \(r\) this is clear in the case of the relative capacity. Let us show that this is the case also for the absolute capacity provided we choose \(c\) small enough (depending only on \(R\) and \(N\), \(\kappa\)). To prove this we use that $v$ cannot be too small outside of $B_R$. More precisely, by comparison
		principle we know that $$v(x)\geq v_{B_R}(x)=1-\frac{R^{N-2}}{\vert x\vert^{N-2}},$$ where $v_{B_R}$
		is the corresponding function for $B_R$.
		Suppose that $B_{\kappa r}(x_0) \backslash B_R\neq\emptyset$. 
		Then the part of $\partial B_r(x_0) \backslash B_R$  with the distance at least $\frac{1-\kappa}{2}r$ from the boundary of the ball $B_R$ has measure at least  $c(\kappa)r^{N-1}$. Then
		$$\fint_{\partial B_r(x_0)}v\geq c(\kappa)\left(1-\frac{R^{N-2}}{\left(R+\frac{1-\kappa}{2}r\right)^{N-2}}\right)\geq c(N,\kappa,R)r,$$ 
		in contradiction with \eqref{e:piccolo} if \(c\) is small enough depending on \(\kappa, N, R\).
		
		Now we turn to the proof of the lemma for both cases. Since \(x_0\) is fixed we  simply write \(B_r\) for \(B_r(x_0)\). The idea is to take as a variation a domain, defined by a function coinciding with $v$ everywhere outside $B_{\sqrt{\kappa}r}$ and being zero inside $B_{\kappa r}$.
		More precisely, define $w$ in $B_{\sqrt{\kappa}r}$ as the solution of
				\begin{equation*}
			\begin{cases}
				\Delta w=0\text{ in }B_{\sqrt{\kappa}r}\backslash B_{\kappa r}\\
				w=0\text{ in }B_{\kappa r}\\
				w=\overline{v}\text{ on }\partial B_{\sqrt{\kappa}r}
			\end{cases},
		\end{equation*}
		where $\overline{v}=\sup_{B_{\sqrt{\kappa}r}}{v}$.
		Note that since $v$ is subharmonic, $\overline{v}\leq c(N,\kappa)\fint_{\partial B_r}{v}$. Moreover, one easily estimates
		\begin{equation}\label{e:normal}
		\Big|\frac{\partial w}{\partial \nu}\Big|\le C(n,\kappa)\frac{\overline{v}}{r} \qquad\text{on \(\partial B_{\kappa r}\).}
		\end{equation}
		 Using the second inequality in Lemma \ref{varcap} with $A=B_{\sqrt{\kappa}r}$ and $\max(u,1-w)=1-\min(v,w)$ in the place of $w$, we get
		\begin{equation*}
			\int_{B_{\sqrt{\kappa}r}}{\vert\nabla v\vert^2}dx+
			\frac{\hat \eta}{2}\vert B_{\sqrt{\kappa}r}\cap\{v>0,w=0\}\vert
			\leq\int_{B_{\sqrt{\kappa}r}}{\vert\nabla\min(v,w)\vert^2}dx.
		\end{equation*}
		Using Cauchy-Schwarz inequality, we obtain
		\begin{equation*}
			\begin{aligned}
				&\int_{B_{\kappa r}}\left(\vert\nabla v\vert^2+\frac{\hat \eta}{2}1_{\{v>0\}}\right)dx
				\leq\int_{B_{\sqrt{\kappa}r}\setminus B_\kappa}\left(\vert\nabla\min(v,w)\vert^2-\vert\nabla v\vert^2\right)dx
				\\
				&\quad\leq 2\int_{(B_{\sqrt{\kappa}r}\backslash B_\kappa)\cap\{v>w\}}\left(\vert\nabla w\vert^2-\nabla v\cdot\nabla w\right)dx
				=-2\int_{B_{\sqrt{\kappa}r}\backslash B_\kappa}\nabla \max(v-w,0)\nabla wdx\\
				&\quad=2\int_{\partial B_{\kappa r}}v\frac{\partial w}{\partial \nu}d\mathcal{H}^{N-1}\leq c(N,\kappa)\frac{\overline{v}}{r}\int_{\partial B_{\kappa r}}v d\mathcal{H}^{N-1}.
			\end{aligned}	
		\end{equation*}
		where we have used \eqref{e:normal}.
		We will now  bound $\int_{\partial B_{\kappa r}}{v}d\mathcal{H}^{N-1}$ from above by a constant times the left-hand side. Since  $\frac{\overline{v}}{r}$ can be made as small as we wish, this will conclude the proof. In order to do that we use first the trace inequality, then Cauchy-Schwarz to get
		\begin{equation*}
			\begin{aligned}
				\int_{\partial B_{\kappa r}}{v}d\mathcal{H}^{N-1}&\leq c(N,\kappa)\left(\frac{1}{r}\int_{B_{\kappa r}}{v}dx+\int_{B_{\kappa r}}{\vert\nabla v\vert}dx\right)\\
				&\leq c(N,\kappa,R)\int_{B_{\kappa r}}\Bigl(\vert\nabla v\vert^2+\frac{\hat \eta}{2}1_{\{v>0\}}\Bigr)dx.
			\end{aligned}	
		\end{equation*}
		
	\end{proof}
	
	\begin{lemma} \label{lip}
		There exists $M=M(N,R)$  such that if \(u_j\) is a minimizer for \eqref{pertpb} and  \(v_j=1-u_j\) satisfies 
		\begin{equation*}
			\fint_{\partial B_r(x_0)}v_j d\mathcal{H}^{N-1}\geq Mr, 
		\end{equation*}
		then $v_j>0$ in $B_r(x_0)$.
	\end{lemma}
	\begin{proof}
		Let us  drop the subscript $j$ as usual. As a comparison domain here we consider $\Omega\setminus B_r(x_0)$, note that it is a subset  of \(B_R\).More precisely, we define $w$ as the solution of
		\begin{equation*} \label{eqforwinBr}
			\begin{cases}
				\Delta w=0\text{ in }B_r(x_0)\\
				w=v\text{ on }\mathds R^N\setminus B_r(x_0).
			\end{cases}
		\end{equation*}
		We  use Lemma \ref{varcap} and Remark \ref{varcapharm} with $A=B_r$, $1-w$ as $w$ to deduce
		\begin{equation} \label{estvzerofrmblw}
			\int_{B_r(x_0)}{\vert\nabla (v-w)\vert^2}dx\leq\left(\frac{1}{\hat \eta}+C\sigma\right)\vert\{v=0\}\cap B_r(x_0)\vert.
		\end{equation}
		We now  estimate $\vert\{v=0\}\cap B_r\vert$ by the left-hand side. This can be done by arguing as in \cite[Lemma 3.2]{AC}. Here we present a slightly different proof \footnote{We warmly thank Jonas Hirsch for suggesting this proof.}. First we change coordinates so that \(x_0=0\). Then   by the representation formula
		\begin{equation}\label{e:harnack}
		w(x)\ge c(N)\frac{r-|x|}{r}\fint_{\partial B_r} v\ge c(N)(r-|x|)M.
		\end{equation}
		If we now apply  Hardy inequality, 
		\[
		\int_{B_r} \frac{g^2}{(r-|x|)^2}\le C(N)\int_{B_r} |\nabla g|^2\qquad g\in W^{1,2}_0(B_r),
		\]
		to the function \(g=v-w\) and we take into account \eqref{e:harnack} and \eqref{estvzerofrmblw}, we get
		\[
		\begin{split}
		c(N)M^2\vert\{v=0\}\cap B_r\vert&\le \int_{\{v=0\}\cap B_r}\frac{w^2}{(r-|x|)^2}\le \int_{ B_r}\frac{(w-v)^2}{(r-|x|)^2}
		\\
		&\le c(N)\int_{B_r} |\nabla (v-w)|^2 \le C(n,R) \vert\{v=0\}\cap B_r\vert,
		\end{split}
		\]
		which is impossible if \(M\) is large enough depending in \(N, R\) unless \(v>0\) almost everywhere in \(B_r\).

	\end{proof}
	
	As in Section 3 of \cite{AC} these two lemmas imply Lipschitz continuity of minimizers and density estimates on the boundary of minimizing domains. Note that we use here Lemma \ref{holder_continuity} as we need to apply the lemmas for the balls of all radii less or equal to some $\rho_0$, see \eqref{e:inc}.
	
	\begin{lemma}
		Let $v_j$ be as above, $\Omega_j=\{v_j=0\}$.
		Then $\Omega_j$ is open and there exist constants $C=C(N, R)$, $\rho_0=\rho_0(N,R)>0$ such that
		\begin{enumerate}[label=\textup{(\roman*)}]
			\item	for every $x\in B_R$
				\begin{equation*}
					\frac{1}{C}\dist(x,\Omega_j)\leq v_j\leq C \dist(x,\Omega_j);
				\end{equation*}
			\item	$v_j$ are equi-Lipschitz;
			\item	for every $x\in\partial\Omega_j$ and $r\leq\rho_0$
				\begin{equation*}
					\frac{1}{C}\leq\frac{\vert\Omega_j\cap B_r(x)\vert}{\vert B_r(x)\vert}\leq\left(1-\frac{1}{C}\right).
\end{equation*}					
		\end{enumerate} 
	\end{lemma}
	
	Applying  \cite[Theorem 4.5]{AC} to \(v_j=(1-u_j)\) we also have the following
	\begin{lemma}\label{conq}
		Let $u_j$ be as above, then there exists a Borel function $q_{u_j}$ such that
		\begin{equation} \label{eqonq}
			\Delta u_j=q_{u_j}\mathcal{H}^{N-1}\mres\partial^*\Omega_j.
		\end{equation}
		Moreover, $0<c\leq -q_{u_j}\leq C$, \(c=c(n,R)\), \(C=C(n,R)\) and $\mathcal{H}^{N-1}(\partial\Omega_j\backslash\partial^*\Omega_j)=0$ .
	\end{lemma}	
	Since $\Omega_j$ converge to $B_1$ in $L^1$ by Lemma \ref{firstpropmin}, the density estimates also give us
	the following convergence of boundaries.
	\begin{lemma}\label{l:conv1}
	Let \(\Omega_j\) be minimizers of \eqref{pertpb}. Then:
		\begin{enumerate}[label=\textup{(\Alph*)}]
		\item For the capacity with respect to the ball $B_R$
			$$\partial\Omega_j\xrightarrow[j\to\infty]{}\partial B_1$$ in the Kuratowski sense.
			\item For the capacity in $\mathds{R}^N$ every limit point of $\Omega_j$ 
			with respect to $L^1$ convergence is the unit ball centered at some $x_\infty\in B_R$.
			Moreover, the convergence holds also in the Kuratowski sense.

		\end{enumerate}
	\end{lemma}
	\begin{cor} \label{closetoball}
	  In the setting of Lemma \ref{l:conv1}, for every $\delta>0$ there exists \(j_\delta\) such that for $j\ge j_\delta$	  \begin{enumerate}[label=\textup{(\Alph*)}]
			\item $B_{1-\delta}\subset\Omega_j\subset B_{1+\delta}$
		in the case of the relative capacity;
		\item $B_{1-\delta}(x_j)\subset\Omega_j\subset B_{1+\delta}(x_j)$ for some $x_j\in B_R$
		in the case of the capacity in $\mathds{R}^n$.
	
	  \end{enumerate}
	\end{cor}

	\subsection{Higher regularity of the free boundary}
	
	In order to address the higher regularity of $\partial\Omega_j$, we need to prove that $q_{u_j}$ is smooth.  This will  be done by using the Euler-Lagrange equations for our minimizing problem. 
	We defined $\Omega_j$ in such a way that the following minimizing property holds
	\begin{enumerate}[label=(\Alph*)] 
		\item 
			\begin{equation}\label{minineq_R}
			\begin{aligned}
				\cp{B_R}{u_j}+f_{\hat{\eta}}(\vert\{u_j=1\}\vert)+\sqrt{\epsilon_j^2+\sigma^2(\alpha_R(\{u_j=1\})-\epsilon_j)^2}\\
				\leq\cp{B_R}{u}+f_{\hat{\eta}}(\vert\{u=1\}\vert)+\sqrt{\epsilon_j^2+\sigma^2(\alpha_R(\{u=1\})-\epsilon_j)^2}
			\end{aligned}	
			\end{equation}	
			for any $u\in W^{1,2}_0(B_R)$ such that $0\leq u\leq 1$.
		\item
			\begin{equation}\label{minineq}
			\begin{aligned}
				\cp{\mathds{R}^N}{u_j}+f_{\hat{\eta}}(\vert\{u_j=1\}\vert)+\sqrt{\epsilon_j^2+\sigma^2(\alpha(\{u_j=1\})-\epsilon_j)^2}\\
				\leq\cp{\mathds{R}^N}{u}+f_{\hat{\eta}}(\vert\{u=1\}\vert)+\sqrt{\epsilon_j^2+\sigma^2(\alpha(\{u=1\})-\epsilon_j)^2}
			\end{aligned}	
			\end{equation}	
			for any $u\in W^{1,2}(\mathds{R}^N)$ such that $0\leq u\leq 1$, $\{u=1\}\subset B_R$.
	\end{enumerate}
	To write Euler-Lagrange equations for $u_j$, we need to have \eqref{minineq_R} or \eqref{minineq})respectively 
	for  $u_j\circ \Phi$  where \(\Phi\) is a diffeomorphism of \(\mathds R^N\) close to the identity. Note that to make sure that \(\{u_j\circ\Phi=1\}\) is contained in \(B_R\) one needs to know that \(\dist(u_j,\partial B_r)>0\).  This follows from Corollary \ref{closetoball}, up  translate \(\Omega_j\) in the case of  the absolute capacity  (note that in this case the problem is invariant by translation). More precisely we will get the following optimality condition
	
	\begin{enumerate}[label=(\Alph*)]
	\item
	\begin{equation*}\label{e:1}
		q_{u_j}^2-\frac{\sigma^2(\alpha_R(\Omega_j)-\epsilon_j)}{\sqrt{\epsilon_j^2+\sigma^2(\alpha_R(\Omega_j)-\epsilon_j)^2}}\vert x\vert=\Lambda_j;
	\end{equation*}
	\item
	\begin{equation*}\label{e:2}
	q_{u_j}^2-\frac{\sigma^2(\alpha(\Omega_j)-\epsilon_j)}{\sqrt{\epsilon_j^2+\sigma^2(\alpha(\Omega_j)-\epsilon_j)^2}}\left(\vert x-x_{\Omega_j}\vert-\left(\fint_{\Omega_j}\frac{y-x_{\Omega_j}}{\vert y-x_{\Omega_j}\vert}dy\right)\cdot x\right)
	=\Lambda_j
	\end{equation*}
	\end{enumerate}
	for some constant $\Lambda_j>0$. These equations are an immediate consequence of the following lemma whose proof is almost the same as \cite[Lemma 4.15]{fkstab} (which in turn is based on~\cite{AguileraAltCaffarelli86}). For this reason we only highlight the most relevant changes, referring the reader to \cite[Lemma 4.15]{fkstab} for more details.
	
	\begin{lemma}\label{l:EL} There exists $j_0$ such that for any $j\geq j_0$ and any two points $x_1$ and $x_2$
	in the reduced boundary of $\Omega_j$ the following equality holds:
	\begin{enumerate}[label=\textup{(\Alph*)}]
	\item
		\begin{equation*}
			q_{u_j}^2(x_1)-\frac{\sigma^2(\alpha_R(\Omega_j)-\epsilon_j)}{\sqrt{\epsilon_j^2+\sigma^2(\alpha_R(\Omega_j)-\epsilon_j)^2}}\vert x_1\vert
			=q_{u_j}^2(x_2)-\frac{\sigma^2(\alpha_R(\Omega_j)-\epsilon_j)}{\sqrt{\epsilon_j^2+\sigma^2(\alpha_R(\Omega_j)-\epsilon_j)^2}}\vert x_2\vert;
		\end{equation*}
	\item	
		\begin{equation*}
		\begin{aligned}
			q_{u_j}^2(x_1)&-\frac{\sigma^2(\alpha(\Omega_j)-\epsilon_j)}{\sqrt{\epsilon_j^2+\sigma^2(\alpha(\Omega_j)-\epsilon_j)^2}}\left(\vert x_1-x_{\Omega_j}\vert-\left(\fint_{\Omega_j}\frac{y-x_{\Omega_j}}{\vert y-x_{\Omega_j}\vert}dy\right)\cdot x_1\right)\\
			&=q_{u_j}^2(x_2)-\frac{\sigma^2(\alpha(\Omega_j)-\epsilon_j)}{\sqrt{\epsilon_j^2+\sigma^2(\alpha(\Omega_j)-\epsilon_j)^2}}\left(\vert x_2-x_{\Omega_j}\vert-\left(\fint_{\Omega_j}\frac{y-x_{\Omega_j}}{\vert y-x_{\Omega_j}\vert}dy\right)\cdot x_2\right).
		\end{aligned}	
		\end{equation*}

	\end{enumerate}
	\end{lemma}	
	\begin{proof}
		We argue by contradiction. Assume there exist $x_1,x_2\in\partial^*\{u_j=1\}$ 
		such that
		\begin{enumerate}[label=(\Alph*)]
		\item
		\begin{equation}\label{contrEL}
			q_{u_j}^2(x_1)-\frac{\sigma^2(\alpha_R(\Omega_j)-\epsilon_j)}{\sqrt{\epsilon_j^2+\sigma^2(\alpha_R(\Omega_j)-\epsilon_j)^2}}\vert x_1\vert
			<q_{u_j}^2(x_2)-\frac{\sigma^2(\alpha_R(\Omega_j)-\epsilon_j)}{\sqrt{\epsilon_j^2+\sigma^2(\alpha_R(\Omega_j)-\epsilon_j)^2}}\vert x_2\vert;
		\end{equation}
		\item		
		\begin{equation}\label{contrEL1}
		\begin{aligned}
			q_{u_j}^2(x_1)&-\frac{\sigma^2(\alpha(\Omega_j)-\epsilon_j)}{\sqrt{\epsilon_j^2+\sigma^2(\alpha(\Omega_j)-\epsilon_j)^2}}\left(\vert x_1-x_{\Omega_j}\vert-\left(\fint_{\Omega_j}\frac{y-x_{\Omega_j}}{\vert y-x_{\Omega_j}\vert}dy\right)\cdot x_1\right)\\
			&<q_{u_j}^2(x_2)-\frac{\sigma^2(\alpha(\Omega_j)-\epsilon_j)}{\sqrt{\epsilon_j^2+\sigma^2(\alpha(\Omega_j)-\epsilon_j)^2}}\left(\vert x_2-x_{\Omega_j}\vert-\left(\fint_{\Omega_j}\frac{y-x_{\Omega_j}}{\vert y-x_{\Omega_j}\vert}dy\right)\cdot x_2\right).
		\end{aligned}	
		\end{equation}
		\end{enumerate}
		Using this inequalities, we are going to construct a variation contradicting \eqref{minineq_R}. We take a smooth radial symmetric function $\phi(x)=\phi(\vert x\vert)$ supported in $B_1$ and define the following diffeomorphism for small $\tau$ and $\rho$:
		\begin{equation*}
			\Phi_\tau^\rho(x)=
			\begin{cases}
				x+\tau\rho\phi(\vert\frac{x-x_1}{\rho}\vert)\nu(x_1), &x\in B_\rho(x_1),\\
				x-\tau\rho\phi(\vert\frac{x-x_2}{\rho}\vert)\nu(x_2), &x\in B_\rho(x_2),\\
				x, &\text{otherwise.}
			\end{cases}
		\end{equation*}
		We define the function $$u^\rho_\tau:=u\circ(\Phi_\tau^\rho)^{-1}$$ and we define a competitor domain $\Omega_\tau^\rho$ as the domain with $u_\tau^\rho$ for capacitary potential, i.e. $$\Omega_\tau^\rho:=\{u^\rho_\tau=1\}.$$
		
		Now we are going to show that for $\tau$ and $\rho$ small enough $\mathscr{C}^*_{\hat{\eta}}(\Omega_\tau^\rho)<\mathscr{C}^*_{\hat{\eta}}(\Omega)$. To do that, we first compute the variation of all the terms involved in $\mathscr{C}^*_{\hat{\eta}}$. 

\medskip		
\noindent
		\textbf{Volume.}
		By arguing as in  \cite[Lemma 4.15]{fkstab} one gets
		\begin{equation*}
		\begin{aligned}
			\vert\Omega_\tau^\rho\vert-\vert\Omega\vert
                         &=\tau\rho^N\left(\int_{\{y\cdot\nu(x_1)=0\}\cap B_1}{\phi\left(\left\vert y\right\vert\right)}
			-\int_{\{y\cdot\nu(x_2)=0\}\cap B_1}{\phi\left(\left\vert y\right\vert\right)}\right)
			+o(\tau)\rho^N+o_\tau(\rho^N)
			\\
			&=o(\tau)\rho^N+o_\tau(\rho^N),
		\end{aligned}
		\end{equation*}
where \(o_\tau(\rho^N)\rho^{-N}\) goes to zero as \(\rho \to 0\) and  \(o(\tau)\) is independent on \(\rho\).

\medskip		
\noindent
\textbf{Barycenter.}(for the case of the capacity in $\mathds{R}^N$). Assume that  that $x_\Omega=0$, as in \cite[Lemma 4.15]{fkstab} one gets,

			\[			x_{\Omega_\tau^\rho}=-\rho^N\tau\frac{x_1-x_2}{\vert\Omega\vert}\left(\int_{\{y_1=0\}\cap B_1}{\phi(\vert y\vert)}\right)+\rho^No(\tau)+o_\tau(\rho^N).
		\]
		
		\medskip
		\noindent
		\textbf{Asymmetry.}
		Again by the very same computations as in \cite[Lemma 4.15]{fkstab} one gets
		\begin{equation*}
		\alpha_R(\Omega_\tau^\rho)-\alpha_R(\Omega)
			=-\rho^N \tau\left(\int_{\{y_1=0\}\cap B_1}{\phi(\vert y\vert)}\right)(\vert x_1\vert-\vert x_2\vert)+o(\tau)\rho^N+o_\tau(\rho^N).
		\end{equation*}	
		In the case of asymmetry $\alpha(\Omega)$ we get an additional term:
		\begin{equation*}
		\alpha(\Omega_\tau^\rho)-\alpha(\Omega)
			=-\rho^N\tau\left(\int_{\{y_1=0\}\cap  B_1}{\phi(\vert y\vert)}\right)\Big(\vert x_1\vert-\vert x_2\vert+\left(\fint_\Omega{\frac{y}{\vert y\vert}dy}\right)\cdot(x_1-x_2)\Big)+o(\tau)\rho^N+o_\tau(\rho^N).
		\end{equation*}.	
		\medskip
		\noindent
		
		\textbf{Dirichlet energy}. Again one can argue as in \cite[Lemma 4.15]{fkstab} to get 
		\begin{equation*}
			\Capa_*(\Omega_\tau^\rho)-\Capa_*(\Omega)\le \tau\rho^N\left(\vert q(x_1)\vert^2-\vert q(x_2)\vert^2\right)\int_{B_1\cap\{y_1=0\}}{\phi(\vert y\vert)}dy+o(\tau)\rho^N+o_\tau(\rho^N).
		\end{equation*}
Combining the above estimates one gets
		\begin{enumerate}[label=(\Alph*)] 
			\item
			\begin{equation*}
			\begin{aligned}
				&\left(\int_{B_1\cap\{y_1=0\}}{\phi(\vert y\vert)}dy\right)^{-1}\frac{\mathscr{C}^R_{\hat{\eta},j}(\Omega_\tau^\rho)-\mathscr{C}^R_{\hat{\eta},j}(\Omega)}{\rho^N}\\
				&=\tau\left(\vert q(x_1)\vert^2-\vert q(x_2)\vert^2-\frac{\sigma^2(\alpha_R(\Omega)-\epsilon_j)}{\sqrt{\epsilon_j^2+\sigma^2(\alpha_R(\Omega)-\epsilon_j)^2}}(\vert x_1\vert-\vert x_2\vert)\right)+o(\tau)+o_\tau(1);
			\end{aligned}	
			\end{equation*}
			\item
			\begin{equation*}
			\begin{aligned}
				&\left(\int_{B_1\cap\{y_1=0\}}{\phi(\vert y\vert)}dy\right)^{-1}\frac{\mathscr{C}_{\hat{\eta},j}(\Omega_\tau^\rho)-\mathscr{C}_{\hat{\eta},j}(\Omega)}{\rho^N}\\
				&=\tau\left(\vert q(x_1)\vert^2-\vert q(x_2)\vert^2
				-\frac{\sigma^2(\alpha(\Omega)-\epsilon_j)}{\sqrt{\epsilon_j^2+\sigma^2(\alpha(\Omega)-\epsilon_j)^2}}\left(\vert x_1\vert-\vert x_2\vert+\left(\fint_\Omega{\frac{y}{\vert y\vert}dy}\right)\cdot(x_1-x_2)\right)\right)\\
				&\qquad+o(\tau)+o_\tau(1).
			\end{aligned}	
			\end{equation*}
		\end{enumerate}
	According to \eqref{contrEL} and \eqref{contrEL1} the quantity in  parentheses is strictly negative. Thus, we get a contradiction with the minimality of $\Omega$ for \(\rho\) and $\tau$ small enough.
	\end{proof}
	
	\begin{lemma}[Smoothness of $q_u$]
		There exist constants $\delta=\delta(N,R)>0$, $j_0=j_0(N,R)$, $\sigma_0=\sigma_0(N,R)>0$
		such that for every $j\geq j_0$, $\sigma\leq\sigma_0$ the functions $q_{u_j}$ belong to $C^\infty(\mathcal{N}_\delta(\partial\Omega_j))$.
		
		Moreover, for every $k$ there exists a constant $C=C(k,N,R)$ such that
		\begin{equation*}
			\Vert q_{u_j}\Vert_{C^k(\mathcal{N}_\delta(\partial\Omega_j))}\leq C
		\end{equation*}
		for every $j\geq j_0$.
	\end{lemma}
	\begin{proof}
	We would like to write an explicit formula for $q_{u_j}$ using Euler-Lagrange equations, namely
	\begin{enumerate}[label=(\Alph*)]
	\item
	\begin{equation}
		q_{u_j}=-\left(\frac{\sigma^2(\alpha_R(\Omega_j)-\epsilon_j)}{\sqrt{\epsilon_j^2+\sigma^2(\alpha_R(\Omega_j)-\epsilon_j)^2}}\vert x\vert+\Lambda_j\right)^\frac{1}{2};
	\end{equation}
	\item
	\begin{equation}
		q_{u_j}=-\left(\frac{\sigma^2(\alpha(\Omega_j)-\epsilon_j)}{\sqrt{\epsilon_j^2+\sigma^2(\alpha(\Omega_j)-\epsilon_j)^2}}\left(\vert x-x_{\Omega_j}\vert-\left(\fint_{\Omega_j}\frac{y-x_{\Omega_j}}{\vert y-x_{\Omega_j}\vert}dy\right)\cdot x\right)
	+\Lambda_j\right)^\frac{1}{2}.
	\end{equation}
	\end{enumerate}
	To do that, we need to show that the quantity in the parenthesis is bounded away from zero.
	Indeed, $q_{u_j}$ is bounded from above and below independently of $j$ and
	\begin{enumerate}[label=(\Alph*)]
	\item
	\begin{equation}
		\left\vert\frac{\sigma^2(\alpha_R(\Omega_j)-\epsilon_j)}{\sqrt{\epsilon_j^2+\sigma^2(\alpha_R(\Omega_j)-\epsilon_j)^2}}\vert x\vert\right\vert\leq C(N,R)\sigma;
	\end{equation}
	\item
	\begin{equation}
		\left\vert\frac{\sigma^2(\alpha(\Omega_j)-\epsilon_j)}{\sqrt{\epsilon_j^2+\sigma^2(\alpha(\Omega_j)-\epsilon_j)^2}}\left(\vert x-x_{\Omega_j}\vert-\left(\fint_{\Omega_j}\frac{y-x_{\Omega_j}}{\vert y-x_{\Omega_j}\vert}dy\right)\cdot x\right)
	\right\vert\leq C(N,R)\sigma.
	\end{equation}
	\end{enumerate}
	Then it follows from the Euler-Lagrange equations that also $\Lambda_j$ is bounded from above and below
	independently of $j$. Thus, for $\sigma$ small enough we can write the above-mentioned explicit formula for $q_{u_j}$ and get the conclusion of the lemma.
	\end{proof}

Now we are ready to apply the results of \cite{AC}. Indeed thanks to Lemma \ref{l:EL},  \(v_j=(1-u_j)\) is a weak solution of the free boundary problem   First, we need to recall the definition of flatness for the free boundary, see  \cite[Definition 7.1]{AC} (here it is applied to \(u=(1-v)\)).

	\begin{defin} 
		Let $\mu_-,\mu_+\in(0, 1]$. A weak solution $u$ of \eqref{eqonq} is said to be of class
$F(\mu_-,\mu_+,\infty)$ in $B_\rho(x_0)$ in a direction $\nu\in S^{N-1}$ if	$x_0\in\partial\{u=1\}$ and
				\[		
				\begin{cases}
					u(x)=1 \qquad&\text{ for }(x-x_0)\cdot\nu\leq-\mu_-\rho,\\
					1-u(x)\geq q_u(x_0)((x-x_0)\cdot\nu-\mu_+\rho) &\text{ for } (x-x_0)\cdot\nu\geq\mu_+\rho,\\
				\end{cases}
		\]
	\end{defin}
	
	We are going to use that flat free boundaries are  smooth (again  we apply \cite[Theorem 8.1]{AC} to \(v=(1-u)\))
	
	\begin{theorem}[Theorem 8.1 in \cite{AC}]\label{regflat}  Let $u$ be a weak solution of \eqref{eqonq}) and assume that $q_u$ is Lipschitz continuous. There are constants $\gamma,\mu_0,\kappa, C$ such that if $u$ is of class $F(\mu, 1, \infty)$ in $B_{4\rho}(x_0)$ in some direction $\nu\in S^{N-1}$ with $\mu\leq\mu_0$ and $\rho\leq\kappa\mu^2$, then there exists a $C^{1,\gamma}$ function $f:\mathds{R}^{N-1}\rightarrow\mathds{R}$ with $\Vert f\Vert_{C^{1,\gamma}}\leq C\mu$ such that
\begin{equation}
	\partial\{u=1\}\cap B_\rho(x_0)=(x_0+\graph_\nu f)\cap B_\rho(x_0),
\end{equation}
where $\graph_\nu f=\{x\in\mathds{R}^N:x\cdot\nu=f(x-x\cdot\nu)\nu)\}$.
Moreover if $q_u\in C^{k,\gamma}$ in some neighborhood of $\{u_j=1\}$, then $f\in C^{k+1,\gamma}$ and 
$\Vert f\Vert_{C^{k+1,\gamma}}\leq C(N,R,\Vert q_u\Vert_{C^{k,\gamma}})$.
	\end{theorem}
	
	We are now ready to prove Theorem \ref{SelPr}, cp. \cite[Proposition 4.4]{fkstab}. 
	\begin{proof}[Proof of Theorem \ref{SelPr}]
		We define $\Omega_j$ as minimizers of \eqref{pertpb}. The desired sequence of Selection Principle
		will be properly rescaled $\{\Omega_j\}$.
		We need to show that $\{\Omega_j\}$ converges smoothly to the ball $B_1$. Indeed one then define 
		\[
		U_j=\lambda_j(\Omega_j-x_*),
		\]
		where \(x_*=0\) in the case of the relative capacity and \(x_*=x_{\Omega_j}\) in the case of the absolute capacity. Theorem \ref{firstpropmin} then implies all the desired properties of \(U_j\), compare with  \cite[Proof of Proposition 4.4]{fkstab}.
		
		Let $\mu_0$, $\kappa$ be as in Theorem \ref{regflat} and $\mu<\mu_0$ to be fixed later. Let $\overline{x}$ be some point on the boundary of $B_1$. As $\partial B_1$ is smooth, it lies inside a narrow strip in the neighborhood of $\overline{x}$. More precisely, there exists $\rho_0=\rho_0(\mu)\leq\kappa\mu^2$ such that for every $\rho<\rho_0$ and every $\overline{x}\in\partial B_1$
		$$\partial B_1\cap B_{5\rho}(\overline{x})\subset\{x: \vert(x-\overline{x})\cdot\nu_{\overline{x}}\vert\leq\mu\rho\}.$$
		
		We know that $\partial\Omega_j$ are converging to $\partial B_1$ in the sense of Kuratowski. Thus, 
		there exists a point $x_0\in\partial\Omega_j\cap B_{\mu\rho_0}(\overline{x})$ such that
		$$\partial \Omega_j\cap B_{4\rho_0}(x_0)\subset\{x: \vert(x-x_0)\cdot\nu_{\overline{x}}\vert\leq 4\mu\rho_0\}.$$
		So, $u_j$ is of class $F(\mu,1,\infty)$ in $B_{4\rho_0}(x_0)$ with respect to the direction $\nu_{\overline{x}}$ and by Theorem \ref{regflat}, $\partial \Omega_j\cap B_{\rho_0}(x_0)$ is the graph of a smooth function with respect to $\nu_{\overline{x}}$. More precisely, for $\mu$ small enough there exists a family of smooth functions $g^{\overline{x}}_j$ with uniformly bounded $C^k$ norms such that
		$$\partial \Omega_j\cap B_{\rho_0}(\overline{x})=\{x+g^{\overline{x}}_j(x)x:x\in\partial B_1\}\cap B_{\rho_0}(\overline{x}).$$
		By a covering argument this gives a family of smooth functions $g_j$ with uniformly bounded $C^k$ norms such that
		$$\partial \Omega_j=\{x+g_j(x)x:x\in\partial B_1\}.$$
		By Ascoli-Arzel\`a and convergence to $\partial B_1$ in the sense of Kuratowski, we get that $g_j\rightarrow 0$ in $C^{k-1}(\partial B_1)$, hence the smooth convergence of $\partial\Omega_j$.
	\end{proof}	
	
	\section{Reduction to  bounded sets}\label{redtobdd}
	To complete the proof of Theorem \ref{mainthmfrnkl} one needs to show that in the case of the full capacity one can just consider  sets with  uniformly bounded diameter. To this end let us introduce the following  
		\begin{defin}
		Let $\Omega$ be an open set in $\mathds{R}^n$ with $\vert\Omega\vert=\vert B_1\vert$.
		Then we define the deficit of $\Omega$ as the difference between its capacity and the capacity of the unit ball:
		$$D(\Omega)=\Capa(\Omega)-\Capa(B_1).$$
	\end{defin}	
	
	Here is the key lemma for reducing Theorem \ref{mainthmfrnkl} to Theorem \ref{mainthmbdd}.	
	
	\begin{lemma}\label{reducetobdd}
		There exist constants $C=C(N)$, $\delta=\delta(N)>0$ and $d=d(N)$ such that for any $\Omega\subset\mathds{R}^n$  open with $\vert\Omega\vert=\vert B_1\vert$ and $D(\Omega)\leq\delta$,
		we can find a new set $\tilde{\Omega}$ enjoying the following properties
		\begin{enumerate}
			\item $\diam(\tilde{\Omega})\leq d$,
			\item $\vert\tilde{\Omega}\vert=\vert B_1\vert$,
			\item $D(\tilde{\Omega})\leq CD(\Omega)$,
			\item \label{boundonasym} $\mathcal{A}(\tilde{\Omega})\geq\mathcal{A}(\Omega)-CD(\Omega)$.
		\end{enumerate}
	\end{lemma}

We are going to define  $\tilde{\Omega}$  as a  suitable dilation  of $\Omega\cap B_S$ for some large \(S\). Hence, we first show the following estimates on the capacity of $\Omega\cap B_S$.
 	
 	\begin{lemma} \label{estcap}
 		Let $S'>S$. Then there exists a constant $c=c(S')$ such that for any open set 	
 		$\Omega\subset\mathds{R}^N$ with $\vert\Omega\vert=\vert B_1\vert$ the following inequalities hold:
 		\begin{equation*}
 			\Capa(B_1)\left(1-\frac{\vert\Omega\setminus B_S\vert}{\vert B_1\vert}\right)^\frac{N-2}{N}\leq \Capa(\Omega\cap B_S)\leq \Capa(\Omega)-c\left(1-\frac{S}{S'}\right)^\frac{N-2}{N}\vert\Omega\setminus B_{S'}\vert^\frac{N-2}{N}.
 		\end{equation*}
 	\end{lemma}
 	\begin{proof}
 		The first inequality is a direct consequence of the classical isocapacitary inequality. To prove the second one we are going to use the estimates for the capacitary potential of $B_S$ for which the exact formula can be written. Denote by $u_{\Omega}$ and $u_S$ the capacitary potentials of $\Omega$ and $\Omega\cap B_S$ respectively. We first compute 		\begin{equation*}
			\begin{aligned}
			&\Capa(\Omega\cap B_S)=\Capa(\Omega)+\int_{\mathds{R}^n}{\vert\nabla u_S\vert^2-\vert\nabla u_\Omega\vert^2}\\
			&=\Capa(\Omega)-\int_{(\Omega\cap B_S)^c}\vert\nabla(u_\Omega-u_S)\vert^2+2\int_{(\Omega\cap B_S)^c}\nabla u_S\cdot\nabla (u_S-u_\Omega)\\
			&=\Capa(\Omega)-\int_{(\Omega\cap B_S)^c}\vert\nabla(u_\Omega-u_S)\vert^2-2\int_{(\Omega\cap B_S)^c}(\Delta u_S) (u_S-u_\Omega)\\
			&\qquad+2\int_{\partial(\Omega\cap B_S)}(u_S-u_\Omega)\nabla u_S\cdot \nu d\mathcal{H}^{N-1}\\
			&=\Capa(\Omega)-\int_{(\Omega\cap B_S)^c}\vert\nabla(u_\Omega-u_S)\vert^2\\			
			\end{aligned}
		\end{equation*}
since \(u_S=u_{\Omega}=1\) on \(\partial(\Omega\cap B_S)\).	We would like to show that $\int_{(\Omega\cap B_S)^c}\vert\nabla(u_\Omega-u_S)\vert^2$ cannot be too small. To this end let us set  $v_\Omega=1-u_\Omega$ and similarly for \(v_S\). By  Sobolev's embedding  we get
	\begin{equation*}
	\begin{aligned}
		&\int_{(\Omega\cap B_S)^c}\vert\nabla(u_\Omega-u_S)\vert^2=\int_{(\Omega\cap B_S)^c}\vert\nabla(v_\Omega-v_S)\vert^2\\
		&\geq c(N)\left(\int_{(\Omega\cap B_S)^c}\vert v_\Omega-v_S\vert^{2^*}\right)^\frac{2}{2^*}\geq c\left(\int_{\Omega\setminus B_S}\vert v_S\vert^{2^*}\right)^\frac{2}{2^*},\\ 
	\end{aligned}	
	\end{equation*}
	where \(2^*\) is the Sobolev exponent and in the last inequality we used  that  $v_\Omega\equiv 0$ on $\Omega$. Let us  also set 
	\[
	z_S=\Biggl(1-\frac{S^{N-2}}{|x|^{N-2}}\Biggr)_+.
	\]
	By the maximum principle, $v_S\geq z_S$, hence
	\begin{equation*}
	\begin{split}
		\int_{\Omega\setminus B_S}\vert v_S\vert^{2^*}&\geq \int_{\Omega\backslash B_S}\vert z_{S}\vert^{2^*}\\
		&\ge \int_{\Omega\backslash B_{S'}}\vert z_{S}\vert^{2^*} \geq \left(1-\left(\frac{S}{S'}\right)^{N-2}\right)^\frac{2N}{N-2}\vert\Omega\setminus B_{S'}\vert.
		\end{split}
	\end{equation*}	
Hence
	\begin{equation*}
		\begin{aligned}
		\Capa(\Omega\cap B_S)&\leq \Capa(\Omega)-c(N)\left(1-\left(\frac{S}{S'}\right)^{N-2}\right)^\frac{N-2}{N}\vert\Omega\setminus B_{S'}\vert^\frac{N-2}{N}\\
		&\leq \Capa(\Omega)-c\left(1-\frac{S}{S'}\right)^\frac{N-2}{N}\vert\Omega\setminus B_{S'}\vert^\frac{N-2}{N},
		\end{aligned}	
	\end{equation*}
	concluding the proof.
 	\end{proof}
	
 We can now prove Lemma \ref{reducetobdd}. 
	\begin{proof} [Proof of Lemma \ref{reducetobdd}]. Let us  assume without loss of generality that the ball achieving the asymmetry of $\Omega$ is $B_1$.  As was already mentioned, we are going to show that there exists an $\tilde{\Omega}$ of the form $\lambda (\Omega\cap B_S)$  for  suitable \(S\) and \(\lambda\) satisfying all the desired properties. Let us set  
	\[
	b_k:=\frac{\vert\Omega\backslash B_{2-2^k}\vert}{\vert B_1\vert}\le 1.
	\]
	Note that by Theorem \ref{t:fmp} we can assume that \(b_1\le 2\mathcal A(\Omega)\) is as small as we wish (independently on \(\Omega\) up to choose \(\delta\) sufficiently small. Lemma \ref{estcap} gives 
	\begin{equation*}
		\begin{aligned}
		\Capa(\Omega)&-c\left(\frac{2^{-(k+1)}}{2-2^{-(k+1)}}\right)^\frac{N-2}{N}b_{k+1}^\frac{N-2}{N}\geq \Capa(B_1)(1-b_k)^\frac{N-2}{N}\geq \Capa(B_1)-\Capa(B_1)b_k,
		\end{aligned}
	\end{equation*}
	which implies
	\begin{equation} \label{estonbk}
		c b_{k+1}\leq 2^k (D(\Omega)+C b_k)^\frac{N}{N-2}.
	\end{equation}
We now claim that there exists \(\bar k\) such that 
\[
b_{\bar k}\le D(\Omega).
\]
Indeed, otherwise   by \eqref{estonbk} we wpuld get
	\begin{equation*}
		b_{k+1}\leq C 2^k (D(\Omega)+C b_k)^\frac{N}{N-2}\leq 2^k C' b_k^\frac{N}{N-2}
		\leq M^k b_k^\frac{N}{N-2}
	\end{equation*}
	for all \(k\in \mathbb N\), where   $M=M(N)$. Iterating the last inequality, we obtain
	\begin{equation*}
		b_{k+1}\leq (M b_1)^{(\frac{N}{N-2})^k}\xrightarrow[k\to\infty]{} 0
	\end{equation*}
	if $b_1$ is small enough, which by Theorem \ref{t:fmp}  we can assume up to choose \(\delta=\delta(N)\ll1\).

	We  define $\tilde{\Omega}$ as a properly rescaled intersection of $\Omega$ with a ball. Let $\bar k $ be such that $b_{\bar k} \le D(\Omega)$
	\begin{equation*}
		\tilde{\Omega}:=\left(\frac{\vert B_1\vert}{\vert\Omega\cap  B_R\vert}\right)^\frac{1}{N}(\Omega\cap  B_R)=(1-b_{\bar k})^{-\frac{1}{N}}(\Omega\cap  B_S),
	\end{equation*}
	where $S:=2-2^{-\bar k}\le 2$. Note that \(|\tilde \Omega|=|B_1|\). We now check all the remaining properties:
	\begin{itemize}
	\item[-] \emph{Bound on the diameter}:
	\begin{equation*}
		\diam(\tilde{\Omega})\leq 2\cdot 2 (1-D(\Omega))^{-\frac{1}{N}}\leq 4(1-\delta)^{-\frac{1}{N}}\le 4.
	\end{equation*}
	up to choose \(\delta=\delta(N)\ll1\).
	\item[-]  \emph {Bound on the deficit}:
	\begin{equation*}
	\begin{aligned}
		D(\tilde{\Omega})&=\Capa(\tilde{\Omega})-\Capa(B_1)=\Capa(\Omega\cap B_S)(1-b_{\bar K})^{-\frac{N-2}{N}}-\Capa(B_1)\\
		&\leq \Capa(\Omega)(1-b_{\bar k})^{-\frac{N-2}{N}}-\Capa(B_1)\\
		&\le  \Capa(\Omega)-\Capa(B_1)+\frac{2(N-2)\Capa(\Omega)}{N}b_{\bar k}\le C(N)D(\Omega).
	\end{aligned}	
	\end{equation*}
	since \(b_{\bar k} \le D(\Omega)\ll 1\) and, in particular, \(\Capa (\Omega)\le 2 \Capa(B_1)\).
	\item[-]\emph{Bound on  the asymmetry}: Let $r:=(1-b_{\bar k})^{-1} \in (1,2)$, that is $r$ is such that $\tilde{\Omega}=r^N(\Omega \cap B_S)$ with \(S=2-2^{-\bar k}\le 2\). Let $x_0$ be such that $B_1(x_0)$ is a minimizing ball for $\mathcal{A}(\tilde{\Omega})$.  Then,  recalling that \(b_{\bar k}=|B_1|^{-1}|\Omega\setminus B_S|\le C(N) D(\Omega)\),
	\begin{equation*}
	\begin{aligned}
		|B_1|\mathcal{A}(\Omega)&\leq\vert\Omega\Delta B_1\left(\frac{x_0}{r}\right)\vert\leq\vert\Omega\setminus B_S\vert+\left\vert(\Omega\cap B_S)\Delta B_1\left(\frac{x_0}{r}\right)\right\vert\\
		&\leq C D(\Omega) +\left\vert(\Omega\cap B_S)\Delta B_\frac{1}{r}\left(\frac{x_0}{r}\right)\right\vert\\
		&\quad+\left\vert B_\frac{1}{r}\left(\frac{x_0}{r}\right)\Delta B_1\left(\frac{x_0}{r}\right)\right\vert\\
		&\le CD(\Omega)+\frac{|B_1|}{r^N}\mathcal{A}(\tilde{\Omega})+\vert B_1\vert\left(1-\frac{1}{r^N}\right)\\
		&\leq CD(\Omega)+|B_1|\mathcal{A}(\tilde{\Omega})+C(N)  b_{\bar k}\\
		&\leq CD(\Omega)+|B_1|\mathcal{A}(\tilde{\Omega}).
	\end{aligned}
	\end{equation*}
	
	\end{itemize}
	\end{proof} 
	
	\section{Proof of Theorem \ref{mainthmfrnkl}}\label{s:proof}
 In order to reduce it to Theorem \ref{mainthmbdd} we need to start with a set which is already close to a ball. In the case of the absolute capacity, thanks to Theorem  \ref{t:fmp},  this can be achieved by assuming the deficit sufficiently small (the quantitative inequality being trivial in the other regime). The next lemma contains the same ``qualitative'' result in the case of the relative capacity.

\begin{lemma}\label{l:qual}
For all \(\varepsilon>0\) there exists \(\delta=\delta(\varepsilon, R)>0\) such that if \(\Omega\subset B_R\) is an open set with \(|\Omega|=1\) and 
\[
\Capa_R(\Omega)\le \Capa_R(B_1)+\delta
\]
then
\[
\alpha_R(\Omega)\le \varepsilon.
\]
\end{lemma}

\begin{proof}		We argue by contradiction. Suppose there exists an $\varepsilon_0>0$ and a sequence of 
		open sets $\Omega_j\subset B_R$ with $\vert\Omega_j\vert=\vert B_1\vert$ such that 
		$\alpha_R(\Omega_j)\geq\varepsilon_0$ but 
		\[
		\Capa_R(B_1)\le \Capa_R(\Omega_j)\le \Capa_R(B_1)+1/j.
		\]
		We denote by $u_j\in W^{1,2}_0(B_R)$ the capacitary potential of $\Omega_j$. The above inequality grants that 
		 \[
		 \int_{B_R}{\vert\nabla u_j\vert^2}dx\rightarrow \Capa_R(B_1).
		 \]
	Thus, up to a not-relabelled subsequence,  there exists a function
		$u$ in $W^{1,2}_0(B_R)$ such that $u_j\rightharpoonup u$ in $W_0^{1,2}(B_R)$, $u_j\rightarrow u$ 
		in $L^2(B_R)$ and almost everywhere in $B_R$. We define $\Omega$ as  $\{u=1\}$. From the lower semi-continuity of Dirichlet integral
		we have that 
		$$\Capa_R(\Omega)\leq\int_{B_R}{\vert\nabla u\vert^2}dx\leq\liminf\int_{B_R}{\vert\nabla u_j\vert^2}dx=\Capa_R(B_1).$$
		 On the other hand, we have $1_\Omega\geq\limsup 1_{\Omega_j}$, meaning that 
		$|\Omega_j\setminus \Omega|\to 0$ and \(|\Omega|\ge  |\Omega_j|=|B_1|$. The  isocapacitary inequality then  implies that  $\Omega=B_1$. In particular, \(|\Omega_j|=|\Omega|\) for all \(j\) and 
		\[
		|\Omega\setminus\Omega_j|=|\Omega_j\setminus \Omega|\to 0,
		\]
		and thus \(1_{\Omega_j}\to1_{\Omega}=1_{B_1}\) in \(L^1(B_R)\). Hence by Lemma \ref{propasym}, \ref{asymlip}, \(\alpha_R(\Omega_j)\to 0\), a contradiction.
			\end{proof}

We have now all the ingredients  to prove  Theorem \ref{mainthmfrnkl}.	
	
 	\begin{proof}[Proof of Theorem \ref{mainthmfrnkl}] We will consider separately the cases of the absolute and relative capacity. 

\medskip
\noindent
\emph{Absolute capacity}. First note that if \(D(\Omega)\ge \delta_0\) then, since \(\mathcal A(\Omega)\ge 2\),
		\[
		D(\Omega)\ge 4\frac{\delta_0}{4}\ge \frac{\delta_0}{4}\mathcal A(\Omega)^2.
		\]
		Hence we can assume that \(D(\Omega)\) is as small as we wish as long as the smallness depends only on \(N\). We now \(\delta_0\) smaller than the constant \(\delta\) in Lemma \ref{reducetobdd} and, assuming that \(D(\Omega)\le \delta_0\), we use Lemma \ref{reducetobdd} to find a set  $\tilde{\Omega}$ with \(\diam(\tilde \Omega)\le d=d(N)\) and satisfying all the properties there. In particular, up to a translation we can assume that \(\tilde \Omega\subset B_d\). Up to choosing  \(\delta_0\) smaller we can apply  Theorem \ref{t:fmp} and Lemma \ref{propasym} \ref{asymlip} to ensure that \(\alpha(\tilde \Omega)\le \epsilon_0\) where \(\epsilon_0=\epsilon_0(N,d)=\epsilon_0(N)$ is the constant appearing in the statement of  Theorem \ref{mainthmbdd}. This, together with Lemma \ref{propasym}, \ref{compasym},  grants that 
		\[
		D(\tilde \Omega)\ge c(N) \alpha(\tilde \Omega)\ge c(N)\mathcal A(\tilde \Omega)^2.
		\]
		Hence, by Lemma \ref{reducetobdd} and assuming that \(\mathcal A(\Omega)\ge C D(\Omega)\) (since otherwise there is nothing to prove),
		\[
		D(\Omega)\ge cD(\tilde \Omega)\ge c A(\tilde \Omega)^2\ge c \mathcal A(\tilde \Omega)^2\ge c \mathcal A(\Omega)^2-CD(\Omega)^2
		\]
		from which the conclusion easily follows since \(D(\Omega)\le \delta_0\ll 1\).
		
 \medskip
\noindent
\emph{Relative capacity}.	Since \(\alpha_R(\Omega)\ \le C(R,N)\) by arguing as in the previous case, we can assume that  \(\Capa_R(\Omega)-\Capa_R(B_1)\le \delta_1(N,R)\ll1\). By Lemma \ref{l:qual} we can assume that \(\alpha_R(\Omega)\le \epsilon_0\) where \(\epsilon_0=\epsilon_0(N,R)\) is the constant in Theorem \ref{mainthmbdd}. Hence
\[
\Capa_R(\Omega)-\Capa_R(B_1)\ge c(N,R)\alpha_R(\Omega)\ge c(N,R) |\Omega\Delta B_1|^2.
\]
	
 	\end{proof}


\begin{thebibliography}{BDPV15}

\bibitem[AAC86]{AguileraAltCaffarelli86}
N.~Aguilera, H.~W. Alt, and L.~A. Caffarelli.
\newblock An optimization problem with volume constraint.
\newblock {\em SIAM J. Control Optim.}, 24(2):191--198, 1986.

\bibitem[AC81]{AC}
H.W. Alt and L.A. Caffarelli.
\newblock Existence and regularity for a minimum problem with free boundary.
\newblock {\em J. Reine Angew. Math.}, 325:105--144, 1981.

\bibitem[ATW93]{AlmgrenTaylorWang93}
Fred Almgren, Jean~E. Taylor, and Lihe Wang.
\newblock {C}urvature-driven flows: a variational approach.
\newblock {\em SIAM J. Control Optim.}, 31(2):387--438, 1993.

\bibitem[BDP17]{BrascoDe-Philippis17}
Lorenzo Brasco and Guido De~Philippis.
\newblock Spectral inequalities in quantitative form.
\newblock In {\em Shape optimization and spectral theory}, pages 201--281. De
  Gruyter Open, Warsaw, 2017.

\bibitem[BDPV15]{fkstab}
L.~Brasco, G.~De~Philippis, and B.~Velichkov.
\newblock Faber--krahn inequalities in sharp quantitative form.
\newblock {\em Duke Math. J.}, 164(9):1777--1831, 2015.

\bibitem[CL12]{CL}
M.~Cicalese and G.~P. Leonardi.
\newblock A selection principle for the sharp quantitative isoperimetric
  inequality.
\newblock {\em Arch. Ration. Mech. Anal.}, 206:617--643, 2012.

\bibitem[Dam02]{Dambrine02}
Marc Dambrine.
\newblock On variations of the shape {H}essian and sufficient conditions for
  the stability of critical shapes.
\newblock {\em RACSAM. Rev. R. Acad. Cienc. Exactas F\'{i}s. Nat. Ser. A Mat.},
  96(1):95--121, 2002.

\bibitem[EG15]{EvansGariepy15}
Lawrence~C. Evans and Ronald~F. Gariepy.
\newblock {\em Measure theory and fine properties of functions}.
\newblock Textbooks in Mathematics. CRC Press, Boca Raton, FL, revised edition,
  2015.

\bibitem[FMP08]{FuscoMaggiPratelli08}
N.~Fusco, F.~Maggi, and A.~Pratelli.
\newblock {The sharp quantitative isoperimetric inequality}.
\newblock {\em Ann. Math.}, 168:941--980, 2008.

\bibitem[FMP09]{nonsharp}
N.~Fusco, F.~Maggi, and A.~Pratelli.
\newblock Stability estimates for certain faber-krahn, isocapacitary and
  cheeger inequalities.
\newblock {\em Annali della Scuola Normale Superiore di Pisa-Classe di
  Scienze}, 8:51--71, 2009.

\bibitem[Fug89]{Fuglede89}
B.~Fuglede.
\newblock {Stability in the isoperimetric problem for convex or nearly
  spherical domains in $\mathbb{R}^n$}.
\newblock {\em Trans. Amer. Math. Soc.}, 314:619--638, 1989.

\bibitem[GM12]{GM}
M.~Giaquinta and L.~Martinazzi.
\newblock {\em An Introduction to the Regularity Theory for Elliptic Systems,
  Harmonic Maps and Minimal Graphs}, volume~11 of {\em Lecture Notes (Scuola
  Normale Superiore)}.
\newblock Edizioni della Normale, 2012.

\bibitem[HHW91]{HHW}
R.R Hall, W.K. Hayman, and A.W. Weitsman.
\newblock On asymmetry and capacity.
\newblock {\em J. d'Analyse Math.}, 56:87--123, 1991.

\bibitem[HN92]{HN}
Wolfhard Hansen and Nikolai Nadirashvili.
\newblock Isoperimetric inequalities for capacities.
\newblock In {\em Harmonic analysis and discrete potential theory ({F}rascati,
  1991)}, pages 193--206. Plenum, New York, 1992.

\bibitem[LL97]{LiebLoss97}
Elliott~H. Lieb and Michael Loss.
\newblock {\em Analysis}, volume~14 of {\em Graduate Studies in Mathematics}.
\newblock American Mathematical Society, Providence, RI, 1997.

\end{thebibliography}

\end{document}